\documentclass[11pt]{article}

\usepackage{amssymb}
\usepackage{amsthm}
\usepackage{amsmath}
\usepackage{amscd}
\usepackage{hyperref}

\allowdisplaybreaks
    
\newcommand{\R}{\mathbb{R}} 
\newcommand{\C}{\mathbb{C}}

\newcommand{\N}{\mathbb{N}}
\newcommand{\Sph}{\mathbb{S}}
\newcommand{\T}{\mathbb{T}}
\newcommand{\eps}{\varepsilon} 

\newcommand{\me}{\mathrm{e}} 
\newcommand{\mi}{\mathrm{i}}
\newcommand{\elbow}{\, \mbox{\rule[-0.1ex]{1.125ex}{0.15ex}\rule[-0.1ex]{0.15ex}{1.25ex}}\;}
\newcommand{\dif}{\mathrm{d}} 
\newcommand{\Dif}{\mathrm{D}}
\renewcommand{\Re}{\mathbf{Re}} 
\renewcommand{\Im}{\mathbf{Im}}

\renewcommand{\det}{\mathrm{det}}

\newcommand{\bPhi}{\boldsymbol{\Phi}}
\newcommand{\torus}{\Sigma_{r}}

\newtheorem*{mainthm}{Main Theorem}
\newtheorem{thm}{Theorem}
\newtheorem{lemma}[thm]{Lemma}
\newtheorem{cor}[thm]{Corollary}
\newtheorem{prop}[thm]{Proposition}

\newtheorem*{nonumcor}{Corollary}

\theoremstyle{definition}

\newtheoremstyle{rmk}{5pt}{5pt}{}{}{\scshape}{:}{.5em}{}
\theoremstyle{rmk}
\newtheorem*{rmk}{Remark}

\newcommand{\mylabel}
	{\label}
\begin{document}

\title{Hamiltonian Stationary Tori  in K\"ahler Manifolds}

\author{
\begin{minipage}{3.125in}
    \begin{center}
        Adrian Butscher \\ Stanford University \\ Department of Mathematics \\ email: \ttfamily butscher@math.stanford.edu
    \end{center}
\end{minipage}
\begin{minipage}{3.125in}
\begin{center}
        Justin Corvino \\ Lafayette College \\ Department of Mathematics \\ email: \ttfamily corvinoj@lafayette.edu
    \end{center}
\end{minipage}\\
\rule{1ex}{0ex}
}

\maketitle

\begin{abstract} 
A Hamiltonian stationary Lagrangian submanifold of a K\"ahler manifold is a Lagrangian submanifold whose volume is stationary under Hamiltonian variations.  We find a sufficient condition on the curvature of a K\"ahler manifold of real dimension four to guarantee the existence of a family of small Hamiltonian stationary Lagrangian tori.  
\end{abstract}


\renewcommand{\baselinestretch}{1.25}
\normalsize

\section{Introduction and Statement of Results}

Let $M^{2n}$ be a K\"ahler manifold with complex structure $J$, Riemannian metric $g$, and symplectic form $\omega$.  The Lagrangian submanifolds of $M$, i.e.~those $n$-dimensional submanifolds of $M$ upon which the pull-back of $\omega$ vanishes, are very natural and meaningful objects to consider  when $M$ is studied from the symplectic point of view.  To gain additional insight by studying $M$ from the metric point of view, it has been fruitful to consider those Lagrangian submanifolds of $M$ which are in some way well-adapted to the metric geometry of $M$. Indeed, it has been found that the Lagrangian submanifolds of $M$ (when $M$ is either K\"ahler-Einstein or Calabi-Yau) that are minimal with respect to the metric $g$ possess a rich mathematical structure and their study is an active area of research (see e.g.~\cite{joycelectures,schoenslagsurvey}).

The minimal and Lagrangian submanifolds of $M$ are critical points of the $n$-dimensional volume functional with respect to compactly supported variations.  It is possible to pose two other natural  variational problems amongst Lagrangian submanifolds of $M$ whose critical points are also mathematically quite interesting.  These variational problems are obtained by restricting the class of allowed variations.  First, one can demand that the volume of  $\Sigma$ is a critical point with respect to only those variations of $\Sigma$ which preserve the Lagrangian condition; in this case, $\Sigma$ is said to be \emph{Lagrangian stationary}.  Since it turns out that a \emph{smooth} Lagrangian stationary submanifold is necessarily minimal (because the mean curvature vector field of $\Sigma$ is itself the infinitesimal generator of a Lagrangian variation, as indicated in \cite{sw1}), points where a  Lagrangian stationary submanifold fails to be minimal must be singular points, and what is of interest is the precise nature of the set of singularities.  A second variational problem that one can pose is the following.  There is a natural sub-class of variations preserving the Lagrangian condition, namely the set of \emph{Hamiltonian transformations}, which are generated by functions on $M$; hence one can also demand that the volume of $\Sigma$ is a critical point with respect to only Hamiltonian variations.  In this case, $\Sigma$ is said to be \emph{Hamiltonian stationary}, and there are indeed examples of non-trivial, smooth, Hamiltonian stationary submanifolds that are not minimal.

Hamiltonian stationary submanifolds of a K\"ahler-Einstein manifold $M$ have been studied by several authors, notably Oh \cite{oh2,oh1},  Helein and Romon \cite{heleinromon2,heleinromon3,heleinromon1}, Schoen and Wolfson \cite{sw1,sw2}.  Oh initially posed the Lagrangian and Hamiltonian stationary variational problems and derived first and second variation formul\ae.  H\'elein and Romon showed that $M$ is a Hermitian symmetric space of real dimension four, this stationarity condition can be reformulated as an infinite-dimensional integrable system whose solutions possess a Weierstra\ss-type representation.  Moreover, they found all Hamiltonian stationary, doubly periodic immersions of $\R^2$ into $\C P^2$ using this representation.  Finally, Schoen and Wolfson initiated the study of Lagrangian variational problems from the geometric analysis point of view, for the purpose of constructing minimal Lagrangian submanifolds as limits of volume-minimizing sequences of Lagrangian submanifolds.  

The approach that is taken in this paper is to state a very general sufficient condition for the existence of a certain type of Hamiltonian stationary submanifold in a K\"ahler manifold $M$.  Namely, we specify a condition at a point $p$ in $M$ which allows us to construct Hamiltonian stationary tori of sufficiently small radii optimally situated in a neighbourhood of the point $p$.  Of course, a simple motivating example is $\C^n$ where one has the standard tori of any radii built with respect to any chosen unitary frame at any chosen point.  These tori will be explicitly used in our construction and will be defined carefully below.  But for a more significant example, we note that all \emph{K\"ahler toric manifolds} contain Hamiltonian stationary Lagrangian tori of the type envisaged  here.  A K\"ahler toric manifold is a closed, connected $2n$-dimensional K\"ahler manifold $(M, g, \omega, J )$ equipped with an effective Hamiltonian holomorphic action $\tau : \T^n \rightarrow \mathrm{Diff} (M)$ of the standard (real) $n$-torus $\T^n$.  The orbits of the group action turn out to be Hamiltonian stationary Lagrangian submanifolds of $M$, essentially because the metric $g$ turns out to be equivariant under the action of $\tau$.   Furthermore, we know that the image of the \emph{moment map} of $\tau$ is a convex polytope in $\R^n$.  If $\mu_\tau : M \rightarrow \R^n$ denotes the moment map and $M_0 := \mu_\tau^{-1} ( \mathit{int} (P) )$ then we know that $M_0$ is an open, dense subset of $M$ that is symplectomorphic to $\mathit{int} (P) \times \T^n$ upon which the action is free.   The orbit tori located near the corners of the polytope turn out to have small volume tending to zero at the corners themselves.  A discussion of the geometry of K\"ahler toric manifolds can be found in \cite{abreu} and the specific example of $\C P^2$ will be presented below for the sake of building intuition.

On the other hand, in a general K\"ahler manifold $M$, one might expect that smooth, small Hamiltonian stationary tori are rather rare, with a condition depending in some way on the ambient geometry of $M$ governing their existence.  The archetype for this kind of a result is an analogous construction of constant mean curvature hypersurfaces in a Riemannian manifold $M$.  Indeed, Ye has shown that it is possible to perturb a sufficiently small geodesic sphere centered at the point $p \in M$ to a hypersurface of exactly constant mean curvature, provided that $p$ is a non-degenerate critical point of the scalar curvature of $M$ \cite{ye}.  

We now explain and state the Main Theorem to be proved in this paper.  Let $ \mathbf{U}_2(M)$ denote the unitary frame bundle of $M$ and choose a point $p \in M$ and a unitary frame $\mathcal U_p \in \mathbf{U}_2(M)$ at $p$.   Let $(z^1, z^2)$ be geodesic normal complex coordinates for a neighbourhood of $p$ whose coordinate vectors at the origin coincide with $\mathcal U_p$.  Fix $r := (r_1, r_2) \in \R^2_+$, the open positive quadrant of $\R^2$, with small $\| r \| $ and define the submanifold
$$\Sigma_{r} (\mathcal U_p) :=  \big\{ \big(r_1 \me^{\mi \theta^1}, r_2 \me^{\mi \theta^2} \big ) : (\theta^1, \theta^2) \in \T^2 \big\} \, .$$
If $M$ were $\C^2$ then $\Sigma_{r}(\mathcal U_p)$ would be Hamiltonian stationary Lagrangian for all $r$ and $ \mathcal U_p $.  In general, $\Sigma_{r}( \mathcal U_p)$ is almost Hamiltonian stationary Lagrangian when $\| r \|$ is very small, as the ambient metric is nearly Euclidean in geodesic normal complex coordinates.  Next, for any section $X \in \Gamma(J(T \Sigma_r(\mathcal U_p)))$ define the deformed submanifold
$$\mu_X \big( \Sigma_{r} (\mathcal U_p) \big) :=  \big\{  \big(r_1(1 + X^1(\theta)) \me^{\mi \theta^1}, r_2 (1 + X^2(\theta)) \me^{\mi \theta^2} \big) : (\theta^1, \theta^2) \in \T^2 \big\} \, .$$

We now want to define a function on unitary frames which will be used to state the existence condition of the Main Theorem below.  First observe that the unitary group acts on $\mathbf{U}_2(M)$ by matrix multiplication in the fiber direction.  The subgroup of diagonal matrices $\mathit{Diag} \subseteq U(2)$ thus acts on $\mathbf{U}_2(M)$ as well, and we define the function $\mathcal F_r : \mathbf{U}_2 (M) / \mathit{Diag}  \rightarrow \R$ by
$$\mathcal F_r( \mathcal U_p) := r_1^2 R^{\C}_{1 \bar 1}(p) + r_2^2 R^{\C}_{2 \bar 2}(p) $$
where $R^{\C}_{1 \bar 1}$ and $R^{\C}_{2 \bar 2}$ are the components of the complex Ricci curvature computed with respect to the chosen frame at the point $p$.  Note that this makes sense since $\mathrm{Ric}^{\C} ( \me^{\mi \alpha} \frac{\partial}{\partial z^j},  \me^{\mi \alpha} \frac{\partial}{\partial z^j}) = \mathrm{Ric}^{\C} ( \frac{\partial}{\partial z^j},   \frac{\partial}{\partial z^j}) $ for all $\alpha \in \Sph^1$.  Furthermore, it is the case that $\Sigma_{r}(\mathcal U_p) = \Sigma_{r}(D \cdot \mathcal U_p)$ for all diagonal matrices $D \in \mathit{Diag}$ so that $\mathcal F_r$ depends on only the information contained in $\mathcal U_p$ that relates to $\Sigma_r(\mathcal U_p)$.

In the statement of the Main Theorem below, the norm $\| \cdot \|_{C^{k,\alpha}_w}$ is a weighted $C^{k,\alpha}$ norm with respect to $g$, defined by 
$$\| u \|_{C^{k, \alpha}_w} := \sup_{ \Sigma_r}  |u| + \|r\| \sup_{ \Sigma_r} \|\nabla u \| + \cdots + \|r\|^k \sup_{ \Sigma_r} \| \nabla^k u \| + \|r\|^{k + \alpha} \big[ \nabla^k u \big]_{\Sigma_r} $$
where $[ \cdot ]_{\Sigma_r}$ is the usual H\"older coefficient on $\Sigma_r$. In addition, we take the metric on the frame bundle to be the natural metric inherited from $g$.  

\begin{mainthm}
	Let $(M, g, \omega, J)$ be a K\"ahler manifold, with $\mathrm{dim}_{ \R}  M=4$.  Suppose $\mathcal U_p \in \mathbf{U}_2(M)$ is such that the equivalence class $[\mathcal U_p] \in \mathbf{U}_2(M) / \mathit{Diag}$ is a non-degenerate critical point of $\mathcal F_r$.  If $\| r \|$ is sufficiently small, then there exists $\mathcal U_{p'} \in \mathbf{U}_2(M) $ and a section $X \in \Gamma( J ( T \Sigma_r( \mathcal U_{p'})))$ so that the submanifold $\mu_X ( \Sigma_r( \mathcal U_{p'}))$ is smooth and Hamiltonian stationary Lagrangian.  Moreover, for any $k\in \N$ and $\alpha\in (0,1)$, we have $\|X \|_{C^{k,\alpha}_w} =\mathcal O( \|r\|^2)$, and the distance between $ \mathcal U_p$ and $ \mathcal U_{p'}$ as points in $\mathbf{U}_2 (M)$ is $\mathcal O(\|r\|^2)$.
\end{mainthm}

We note as a direct corollary that it is possible to extend the Main Theorem slightly in order to answer a more general question.  That is, the Main Theorem finds a Hamiltonian stationary submanifold that is a small perturbation of $\Sigma_r$ for $\| r \|$ sufficiently small.  Now one can ask if it is possible to find neighbouring Hamiltonian stationary Lagrangian submanifolds which are perturbations of $\Sigma_{r'}$ with $r'$ sufficiently close to $r$.  The answer to this question is that one can indeed find such submanifolds because the non-degenerate critical points of the family of functionals $\mathcal F_{r'}$ with $r'$ varying in a neighbourhood of $r$ are \emph{stable}.  That is, if $r'$ is sufficiently close to $r$ then $\mathcal F_{r'}$ has a non-degenerate critical point $[ \mathcal U_{p(r')}] $ near $[\mathcal U_p]$.  By the Implicit Function Theorem, moreover, the association $r' \mapsto [ \mathcal U_{p(r')}] $ is smooth and this can be lifted to a smooth association $r' \mapsto \mathcal U_{p(r')}$

\begin{nonumcor}
	Let $r := (r_1 , r_2) \in \R^2_+$ with $\|r \|$ sufficiently small and suppose $\mathcal U_p \in \mathbf{U}_2(M)$is such that the equivalence class $[\mathcal U_p] \in \mathbf{U}_2(M) / \mathit{Diag}$ is a non-degenerate critical point of $\mathcal F_r$.  Then one can find a small neighbourhood $\mathcal V \subset \R^2_+$ containing $r$ so that $\Sigma_{r'}(\mathcal U_p)$ can be perturbed into a Hamiltonian stationary Lagrangian submanifold of $M$ for all $r' \in \mathcal V$.  Moreover, the mapping taking $r'$ to the associated Hamiltonian stationary Lagrangian submanifold is smooth.
\end{nonumcor}

The Main Theorem will be proved following broadly similar lines as the proof of Ye's result.  That is, for each $\mathcal U_p$ and sufficiently small $\|r \|$, a section $X$ will be found so that $\mu_X ( \Sigma_r( \mathcal U_{p}))$ is \emph{almost} Lagrangian and Hamiltonian stationary; in fact the small error will be arranged to lie in a certain finite-dimensional space.  The discrepancy comes from the fact that the Hamiltonian stationary differential operator possesses an approximate co-kernel (coming from translation and $U(2)$-rotation) that constitutes an obstruction to solvability.  Only when $ \Sigma_r( \mathcal U_{p})$ is very special (such that the image of the Hamiltonian stationary differential operator acting on $ \Sigma_r( \mathcal U_{p})$ is orthogonal to the associated co-kernel to lowest order in $\| r \|$) can a solution be found.  The existence condition, as indicated in the Main Theorem, is that $[\mathcal U_p]$ is a non-degenerate critical point of $\mathcal F_r$.  This condition is qualitatively similar to Ye's condition in that it involves the ambient curvature tensor of $M$.  But of course the condition here takes into account the freedom to choose the complex frame with respect to which $\Sigma_r(\mathcal U_p)$ is built as well as the point $p$ where $\Sigma_r(\mathcal U_p)$ is located.

As with Ye's condition, it is not always the case that $\mathcal F_r$ possesses non-degenerate critical points.  For example, this occurs in the case of $\C P^2$ and of  $\mathbb{C}^2$, despite the fact that both spaces contain small Hamiltonian stationary Lagrangian tori.  These examples can be seen as analogues of the situation in $\R^n$, a space which fails to satisfy the non-degeneracy criterion of Ye and where constant mean curvature spheres come in great abundance.  
It should be noted that Pacard and Xu have recently strengthened Ye's result by replacing the non-degeneracy condition appearing there with a different condition, from which they can deduce that \emph{every} compact Riemannian manifold must have at least one point $p$ for which sufficiently small geodesic spheres centered at $p$ can be perturbed to hypersurfaces of constant mean curvature \cite{pacardxu}.  A similar strengthening should be possible in the Hamiltonian stationary Lagrangian case as well.

\bigskip \noindent \scshape Acknowledgements: \upshape The authors would like to thank Richard Schoen for proposing this problem as well as Rafe Mazzeo and Frank Pacard for their interest and suggestions.  The first author was partially supported by the NSERC Discovery Grant 262453-03.  The second author was partially supported by N.S.F.~Grant DMS-0707317, and by the Fulbright Foundation; he also thanks the hospitality of the Mittag-Leffler Institute, at which part of the work was done.

\section{Geometric Preliminaries}

\subsection{K\"ahler Manifolds}
\label{sec:kaehlergeom}

A complex manifold $M$ of real dimension $2n$ and integrable complex structure $J$  is said to be \emph{K\"ahler} if it possesses a Riemannian metric $g$ for which $J$ is an isometry, as well as a symplectic form $\omega$ satisfying the compatibility condition $\omega(X, Y) = g (JX, Y)$ for all tangent vectors $X, Y$.  Standard references for K\"ahler manifolds are \cite{griffithsharris} and \cite{kobayashinomizu}.   What follows is a brief description, for the purpose of fixing terminology and notation, of those aspects of K\"ahler geometry that will be relevant for what follows.  

The question of interest is the nature of the \emph{local} geometry of a K\"ahler manifold.  Consider first the simplest example of a K\"ahler manifold: this is $\C^n$ equipped with the standard Euclidean metric $\mathring g := \Re\big(\sum_{k} \dif z^k \otimes \dif \bar z^k \big)$ and the standard symplectic form $\mathring \omega :=  - \Im\big(\sum_{k} \dif z^k \otimes \dif \bar z^k \big)$ (both given in complex coordinates), as well as the standard complex structure (which coincides with multiplication by $\sqrt{-1}$ in complex coordinates).   In a general K\"ahler manifold, it is a fact that it is always possible to find local complex coordinates for a neighbourhood $\mathcal V$ of any point $p \in M$  in which the complex structure is standard everywhere in $\mathcal V$, and the metric and symplectic form are standard at $p$ with vanishing derivatives.  In fact, more is true: the metric and symplectic form possess special structure in such a coordinate chart.  

It is possible to show that there is a function $F : \mathcal V \rightarrow \R$, called the \emph{K\"ahler potential}, so that the metric and symplectic form are: 
\begin{align*}
	g = 2\, \Re \sum_{k,l}  \left( \frac{\partial^2 F}{\partial z^k \partial \bar z^l} \dif z^k \otimes \dif \bar z^l \right) &= \frac{1}{2}\sum_{k,l}  \left(\frac{\partial^2 F}{\partial x^k \partial x^l} +  \frac{\partial^2 F}{\partial y^k \partial y^l}\right) \! \big( \dif x^k \otimes \dif x^l + \dif y^k \otimes \dif y^l \big) \\
	&\qquad + \frac{1}{2}\sum_{k,l}  \left(\frac{\partial^2 F}{\partial y^k \partial x^l} -  \frac{\partial^2 F}{\partial x^k \partial y^l}\right) \! \big( \dif y^k \otimes \dif x^l - \dif x^k \otimes \dif y^l \big) \\[1ex]
	\omega = -2 \, \Im \sum_{k,l} \left( \frac{\partial^2 F}{\partial z^k \partial \bar z^l} \dif z^k \otimes \dif \bar z^l \right) &= \frac{1}{2}\sum_{k,l}  \left(\frac{\partial^2 F}{\partial x^k \partial x^l} +  \frac{\partial^2 F}{\partial y^k \partial y^l}\right) \! \big(  \dif x^k \otimes \dif y^l - \dif y^k \otimes \dif x^l \big) \\
	&\qquad + \frac{1}{2}\sum_{k,l}  \left(\frac{\partial^2 F}{\partial y^k \partial x^l} -  \frac{\partial^2 F}{\partial x^k \partial y^l}\right) \! \big( \dif x^k \otimes \dif x^l + \dif y^k \otimes \dif y^l \big) \, ,
\end{align*} 
in the local complex coordinates $(z^1, \ldots, z^n)$ or local real coordinates $(x^1, \ldots, x^n, y^1, \ldots, y^n)$ for $\mathcal V$, which are related by $z^k = x^k + \mi y^k$.  Note that
$$ \omega =\frac{1}{2} \sum_k \dif \! \left(  \frac{\partial F}{\partial x^k } \dif y^k - \frac{\partial F}{\partial y^k } \dif x^k \right),$$
which is consistent with the fact that $\dif \omega = 0$, and locally, closed forms are exact.  Write $\omega := \dif \alpha$, where $\alpha$ is called the Liouville form of $\omega$, and write $\mathring \alpha := \frac{1}{2} \sum_k\big( x^k \dif y^k - y^k \dif x^k \big)$ for the Liouville form of the standard symplectic form.  Note also that the K\"ahler potential is unique up to the addition of a function $\varphi$ satisfying $\partial_{z^k} \partial_{\bar z^l} \varphi = 0$ for all $k, l$.  One can additionally show that it is possible to choose $F$ near the origin having the form 
$$F(z,\bar z) := \frac{1}{2} \|z \|^2 + \hat F(z, \bar z) $$
where $\hat F$ vanishes at least to order four in $z$ and $\bar z$.  Hence $\partial_{z^k} \partial_{\bar z^l} F = \delta_{kl} + \mathcal O(\| z \|^2)$.  Consequently, the K\"ahler structures near the origin are perturbations of the standard structures $\mathring g$ and $\mathring \omega$, whose K\"ahler potential is $\mathring{F}(z, \bar{z}) := \frac{1}{2} \| z \|^2$.

The complexified curvature tensor of a K\"ahler manifold in local coordinates in $\mathcal V$  can be expressed in terms of the K\"ahler potential.  Namely, the complexified curvature tensor satisfies
$$R^{\C}_{k \bar l m \bar n} = \frac{\partial^4 \hat F}{\partial z^k \partial \bar z^l \partial z^m \partial \bar z^m} - \sum\limits_{u,v}g^{\bar u  v} \frac{\partial^3 \hat F}{\partial z^k \partial \bar z^u \partial z^m} \frac{\partial^3 \hat F}{\partial \bar z^l \partial z^v \partial \bar z^n} \, .$$ 
Since $\partial^3 F(0) = 0$, then we have
\begin{equation}
	\label{eqn:cxcurvature}
	R^{\C}_{k \bar l m \bar n}(p) = \frac{\partial^4 \hat F(0)}{\partial z^k \partial \bar z^l \partial z^m \partial \bar z^m}  \, .
\end{equation}

\subsection{Hamiltonian Stationary Lagrangian Submanifolds}

Interesting submanifolds of a K\"ahler manifold can be characterized by the effect of the action of $J$ on tangent spaces.  For instance, a \emph{complex} submanifold of $M^{2n}$ is one whose tangent spaces are invariant under $J$.  Two classes of submanifolds of importance in this paper are defined in terms of a complementary condition to that of a complex submanifold.  An $n$-dimensional submanifold $\Sigma$ is called \emph{Lagrangian} if $J (T_p \Sigma)$ is orthogonal to $T_p\Sigma$ for each $p \in \Sigma$.  Hence a Lagrangian submanifold satisfies $\omega(X, Y) = 0$ for all $X, Y \in T_p \Sigma$ and $p \in \Sigma$.  More generally, an $n$-dimensional submanifold $\Sigma$ for which $J (T_p \Sigma)$ is transverse to $T_p \Sigma$ for each $p \in \Sigma$ is called \emph{totally real}.

We will be interested in diffeomorphisms of $M$ that preserve some or all aspects of its K\"ahler structure.   The diffeomorphisms which preserve the full K\"ahler structure are the \emph{holomorphic isometries} and are quite rare in general.  In $\C^n$, though, there are holomorphic isometries: these are the $U(n)$-rotations.  The diffeomorphisms which preserve the symplectic form but not necessarily the metric are called \emph{symplectomorphisms}.  Every K\"ahler manifold possesses symplectomorphisms; indeed, for each function $u: M \rightarrow \R$ the one-parameter family of diffeomorphisms obtained by integrating the vector field $X$ defined by $X \elbow \omega := \dif u$ are symplectomorphisms.  These diffeomorphisms are called \emph{Hamiltonian}.  The condition of being totally real or Lagrangian is preserved by symplectomorphisms.

Consider now a Lagrangian submanifold $\Sigma \subset M$.  If $\Sigma$ is a critical point of the $n$-dimensional volume functional amongst \emph{all} possible compactly supported variations, then $\Sigma$ is \emph{minimal}, in which case the mean curvature vector $\vec H_\Sigma$ of $\Sigma$ vanishes.  Suppose, however, that $\Sigma$ is merely a critical point of the $n$-dimensional volume amongst only Hamiltonian variations, and thus is Hamiltonian stationary Lagrangian.  By computing the Euler-Lagrange equations for $\Sigma$, it becomes clear that being Hamiltonian stationary is in general a strictly weaker condition than being minimal.  Indeed, let $\phi_t$ be a one-parameter family of Hamiltonian diffeomorphisms of $M$ with infinitesimal deformation vector field $X$ satisfying $X \elbow \omega = \dif u$ for $u : M \rightarrow \R$.  Then 
\begin{align}
	\label{eqn:firstvar}
	0 = \left. \frac{\dif}{\dif t} \mathit{Vol} \big( \phi_t(\Sigma) \big) \right|_{t=0} &= - \int_{\Sigma} g ( \vec H_\Sigma, X) \, \dif \mathrm{Vol}_{\Sigma} \notag \\
	&= - \int_{\Sigma} \omega ( X, J \vec H_\Sigma) \, \dif \mathrm{Vol}_{\Sigma} \notag \\
	&= - \int_{\Sigma} g ( \bar \nabla u , J \vec H_\Sigma) \, \dif \mathrm{Vol}_{\Sigma} \notag \\
	&= \int_{\Sigma} u \, \nabla \! \cdot \! \big( J \vec H_\Sigma  \big) \, \dif \mathrm{Vol}_{\Sigma}
\end{align} 
by Stokes' Theorem.  Here $\bar \nabla$ is the connection associated with the ambient metric $g$ while $\nabla$ is the induced connection of $\Sigma$, and $\nabla \cdot$ is the divergence operator.  Since \eqref{eqn:firstvar} must hold for all functions $u$, it must be the case that the mean curvature of $\Sigma$ satisfies
\begin{equation}
	\label{eqn:hamstateqn}
	\nabla \! \cdot \! \big( J \vec H_\Sigma  \big) = 0 \, .
\end{equation}
Equation \eqref{eqn:hamstateqn} will be solved in this paper to find Hamiltonian stationary Lagrangian submanifolds.

Observe that since $\Sigma$ is Lagrangian and $\vec H_\Sigma$ is normal to $\Sigma$, then $J \vec H_\Sigma$ is tangent to $\Sigma$ and taking its divergence with respect to the induced connection makes sense.  It is convenient to introduce some notation at this point so that the mean curvature (and second fundamental form) of a totally real submanifold can be treated in a similar manner.  To this end, let $\Sigma$ be totally real and define the \emph{symplectic second fundamental form} and the \emph{symplectic mean curvature} of $\Sigma$ by the formul\ae\
$$B(X, Y, Z) := \omega \big( ( \nabla_X Y )^\perp, Z \big) \qquad \mbox{and} \qquad H(Z) := \mathrm{Trace} \big( B( \cdot, \cdot, Z) \big) $$
where $X^\perp$ denotes the orthogonal projection of a vector $X$ defined at a point $p \in \Sigma$ to the normal bundle of $\Sigma$ at $p$.  The symplectic mean curvature is thus a one-form on $\Sigma$ and equation \eqref{eqn:hamstateqn} becomes $\nabla \cdot H = 0$, where again $\nabla \cdot$ is the divergence operator. 

\begin{rmk}
The following observation about the symplectic second fundamental form is important.  If $\Sigma$ is Lagrangian then $B(X, Y, Z) = \omega (\nabla_X Y , Z)$ for all vector fields $X, Y, Z$ tangent to $\Sigma$ since $\omega ( (\nabla_X Y)^\| , Z ) = 0$.  Hence we have the usual symmetry $B(X, Y, Z) = B(Y, X, Z)$.  In addition, we have $B(X, Z, Y) = g (J \nabla_X Z , Y) = g (\nabla_X J Z , Y) = - g (JZ, \nabla_X Y) = g( J \nabla_X Y, Z) = B(X, Y, Z)$.  Consequently the symplectic fundamental form of a Lagrangian submanifold is fully symmetric in all of its slots.
\end{rmk}

\subsection{Hamiltonian Stationary Lagrangian Submanifolds in $\C P^2$}

\label{sec:example}

We now discuss a simple example demonstrating that the K\"ahler manifold $\C P^2$, equipped with the Fubini-Study metric, contains a two-parameter family of Hamiltonian stationary Lagrangian tori that are not minimal; and that there are members of this family with arbitrary small radii.  Therefore $\C P^2$ contains Hamiltonian stationary Lagrangian submanifolds of the type we intend to construct in this paper.  As mentioned in the introduction, the existence of these tori is expected because $\C P^2$ is a toric K\"ahler manifold. 

The family of tori in question will be obtained by projecting a family of three-dimensional tori in $\Sph^5$ to $\C P^2$ using the Hopf projection.  These are found by choosing three positive real numbers $r_1$, $r_2$ and $r_3$ satisfying $r_1^2 + r_2^2 + r_3^2 = 1$, and then setting
$$T_r := \big\{ \big( r_1 \me^{\mi \theta^1}  , r_2 \me^{\mi \theta^2}  , r_3 \me^{\mi \theta^3}  \big) : \theta^k \in \Sph^1 \big\} \, .$$
Here we denote $r := (r_1, r_2, r_3)$.  Notice that $T_r$ is foliated by the Hopf fibration: the fiber through the point $p := ( r_1 \me^{\mi \theta^1}  , r_2 \me^{\mi \theta^2}  , r_3 \me^{\mi \theta^3} )$ is $\{ \me^{\mi \alpha} p : \alpha \in \Sph^1 \}$ which is clearly a subset of  $T_r$.  Moreover, this foliation is regular and thus $\Sigma_r  := \pi_{\mathit{Hopf}} ( T_r )$ is a two-dimensional submanifold of $\C P^2$, where $\pi_{\mathit{Hopf}}:\Sph^5 \rightarrow \C P^2$ is the Hopf projection.  Furthermore, it is clear that $\Sigma_r $ is a torus.

The torus $\Sigma_r $ is Hamiltonian stationary for the following reasons.   First, recall the relationship between the symplectic form $\omega$ of $\C P^2$ and the K\"ahler structure of $\C^3$.  That is, if $V_1$ and $V_2$ are two tangent vectors of  $\C P^2$, then $\omega ( V_1, V_2) := \Re \big( \mathring g ( \mi \hat V_1 , \hat V_2) \big)$ where $\mathring g$ is the Euclidean metric of $\C^3$ and $\hat V_i$ is the unique vector in $(\pi_{\mathit{Hopf}})_\ast ^{-1} (V_i)$ that is orthogonal to the Hopf fiber.  It follows that $\Sigma_r $ is Lagrangian because if $V_i$ is tangent to $\Sigma_r $ then 
$$\hat V_i \in \mathit{span}_{\R} \left\{ \mi z^1 \frac{\partial }{\partial z^1} \, , \,  \mi z^2 \frac{\partial }{\partial z^2} \, , \,  \mi z^3 \frac{\partial }{\partial z^3} \right\}$$
and it is clear that $\Re ( \mathring g ( \mi X, Y ) ) = 0$ for all vectors $X, Y$ belonging to this space.  Next, determining if $\Sigma_r $ is Hamiltonian stationary requires computing its second fundamental form.  Now because $\Sigma_r $ is Lagrangian, it can be lifted to a Legendrian submanifold $\hat \Sigma_r \subseteq T_r$ of $\Sph^5$ and this lifting is a local isometry.  Furthermore, the second fundamental form of $\hat \Sigma_r $  coincides with the second fundamental form of $\Sigma_r$.  Therefore it suffices to compute the second fundamental form of $\hat \Sigma_r $, which is a slightly simpler task and is done as follows.  We can locally parametrize  $\hat v$ by 
$$\mathcal A : (\alpha^1, \alpha^2 ) \mapsto \big( r_1 \me^{\mi L^1(\alpha)}  , r_2 \me^{\mi L^2(\alpha)}  , r_3 \me^{\mi L^3(\alpha)}  \big)$$
where $L^k (\alpha) := \sum_s L^k_s \alpha^s$ is a linear function of $\alpha:= (\alpha^1, \alpha^2)$ chosen so that the tangent vectors $V_s := \mathcal A_\ast \big( \frac{\partial}{\partial \alpha^s} \big)$ are linearly independent and $\sum_{k=1}^3 r_k^2 L^k_s = 0$ for $s = 1, 2$.  This latter condition says that each $V_s$ is orthogonal to the Hopf vector field.  Furthermore, one can check that any other choice of linear functions satisfying the aforementioned constraints amounts to a reparametrization of $\hat \Sigma_r$.   The induced metric of the parametrization is
$$\mathring h := \sum_{s, t} \Re \big( \mathring g (\hat V_s, \hat V_t) \big) \dif \alpha^s \otimes \dif \alpha^t = \sum_{s, t} \left( \sum_{k=1}^3 r_k^2 L^k_s L^k_t \right) \dif \alpha^s \otimes \dif \alpha^t$$
which is a flat metric.  The second fundamental form of this parametrization can be deduced from
$$\Re \big( \mathring g( \mathring \nabla_{V_s} V_t , \mi V_u )\big) = \sum_k r_k^2 L^k_s L^k_t L^k_u$$
where $\mathring \nabla$ is the Euclidean connection, which shows in particular that the second fundamental form is parallel with respect to the induced metric.  Hence its divergence is zero.  Consequently $\Sigma_r $ is Hamiltonian stationary but not minimal.

Finally we would like to know the geometric dimensions of $\Sigma_r $ in $\C P^2$.  Since we know the induced metric of $\Sigma_r $, this amounts to finding the size of the smallest domain in $\hat \Sigma_r $ that maps bijectively onto $\Sigma_r $ under $\pi_{\mathit{Hopf}}$.  After some work, we find that this domain is the parallelogram in the $(\alpha^1, \alpha^2)$-coordinates spanned by the vectors
$$E_k := \sum_{s, t} \mathring h^{st} \, \Re \left(  \mathring g \left( \mi z^k \frac{\partial }{\partial z^k}, \hat V_t \right) \right) \frac{\partial}{\partial \alpha^s} \qquad k = 1, 2 \, .$$
One can check that the volume of this parallelogram with respect to the induced metric $\mathring h$ is given by $r_1 r_2 \sqrt{1- r_1^2 - r_2^2}$.  Hence one can consider $\Sigma_r $ to be small when $r_1$ or $r_2$ tends to zero.

\section{Constructing the Approximate Solution}
\label{sec:con}

Let us assume in this paper from now on that the real dimension of the ambient manifold is four and thus that the dimension of the Hamiltonian stationary Lagrangian submanifold is two, since this simplifies the presentation of the results and their proofs.  We expect that most of the forthcoming calculations should generalize to higher dimensions and similar results will hold. 

\subsection{Rescaling the Ambient Manifold}
\label{sec:scale}

Choose a point $p \in M$ and find local complex coordinates so that a small neighbourhood $\mathcal V$ of $p$ maps to a small neighbourhood $\mathcal V_0$ of the origin in $\C^2$.  Moreover, let these coordinates be such that the metric and symplectic form are of the type discussed in Section \ref{sec:kaehlergeom}.   Assume that the diameter of this neighbourhood is $\rho_0 \in (0,1)$; let $r = (r_1, r_2)$, with $\| r \| < \rho_0$, be the radii of the Hamiltonian stationary Lagrangian torus we intend to construct, and set $\rho := \| r \|$.  Now change coordinates in this neighbourhood and also re-scale the metric and symplectic form via 
\begin{equation}
	\label{eqn:rescale}
	z \stackrel{\varphi}{\mapsto} \rho z \qquad \mbox{and} \qquad g \mapsto \rho^{-2} \varphi^* g \qquad \mbox{and} \qquad \omega \mapsto  \rho^{-2} \varphi^*\omega \, .
\end{equation}
As a result, we obtain a new K\"ahler metric on a large neighbourhood $\| r \|^{-1} \mathcal V_0$ of the origin in $\C^2$, where the complex structure is standard and the K\"ahler potential is 
$$F_{\rho}(z, \bar z) := \frac{1}{2} \| z \|^2 + \rho^2 \hat F_\rho (z, \bar z)$$
with $\hat F_\rho(z, \bar z) := \rho^{-4} \hat F(\rho z, \rho \bar z )$.  Furthermore, the Hamiltonian stationary Lagrangian condition is unchanged under this re-scaling and the torus $\frac{1}{\rho} \Sigma_r$ has radii $(r_1, r_2)$ satisfying $r_1^2 + r_2^2 = 1$.  Therefore, in order to construct a Hamiltonian stationary Lagrangian torus of small radii near $p$, it is sufficient to construct a Hamiltonian stationary Lagrangian torus with unit radius vector near the origin in $\C^2$ with K\"ahler potential $F_\rho$, but to take $\rho$ sufficiently small. Finally, the weighted $C^{k,\alpha}$ norm used in the statement of the Main Theorem is equivalent to the standard $C^{k,\alpha}$ norm under the re-scaling.  

\begin{rmk}
The advantage of working with these scaled coordinates is that it is now possible to express the deviation of the background geometry from Euclidean space very efficiently using the parameter $\rho$. In particular, $\hat F_\rho$ can be expanded in a power series in $z$ and $\bar z$ starting at order four that has coefficients depending on $\rho$ but bounded uniformly by a constant of size  $\mathcal O(\rho^2)$.  
\end{rmk}

\subsection{The Approximate Solution}

Let $ \mathbf{U}_2(M)$ denote the unitary frame bundle of $M$ and choose a point $p \in M$ and a unitary frame $\mathcal U_p \in \mathbf{U}_2(M)$ at $p$.   Let $(z^1, z^2)$ be geodesic normal complex coordinates for a neighbourhood of $p$ whose coordinate vectors at the origin coincide with $\mathcal U_p$.  Now let $r := (r_1, r_2)$ be some fixed vector belonging to $\R_+^2$, the open positive quadrant of $\R^2$, with $\| r \| = 1$.  Define the $2$-dimensional submanifold of $\C^2$ given by
$$\Sigma_r (\mathcal U_p) := \big\{ \big(r_1 \me^{\mi \theta^1}, r_2 \me^{\mi \theta^2} \big ) : (\theta^1, \theta^2) \in \T^2 \big\} \, .$$
Note that $\Sigma_r (\mathcal U_p)$ is the image of the $\T^n$ under the embedding  $\mu_0 : (\theta^1, \theta^2 ) \mapsto \big( r_1 \me^{\mi \theta^1}, r_2 \me^{\mi \theta^2} \big)$.   We will denote $\Sigma_r := \Sigma_r(\mathcal U_\rho)$ when it is not necessary to speak explicitly of the frame $\mathcal U_p$ from which $\Sigma_r(\mathcal U_p)$ is built.  

The following result motivates the use of $\Sigma_r$ as an \emph{approximate solution} of the problem of finding Hamiltonian stationary Lagrangian submanifolds in arbitrary K\"ahler manifolds.

\begin{lemma}
	\label{lemma:torus}
	The submanifold $\Sigma_r$ is Hamiltonian stationary Lagrangian with respect to the standard K\"ahler structure $(\mathring g, \mathring \omega, J)$ of $\C^2$.   In fact, the symplectic second fundamental form $\mathring B$ and the symplectic mean curvature $\mathring H$ are parallel.  
\end{lemma}

\begin{proof}
	We include this standard calculation for the convenience of the reader.  To begin, the tangent vectors of $\Sigma_r$ can be found by differentiating in $\theta$.  In complex notation, these are $E_k :=   \mi r_k    \me^{\mi \theta^k} \frac{\partial}{\partial z^k}$, for $k=1, 2$.  From this we can immediately compute the components of the induced metric $\mathring h$ and those of $\mathring \omega$ restricted to $\Sigma_r$.  Indeed, since the K\"{a}hler potential is $\mathring{F}(z,\bar{z})=\frac{1}{2}\|z\|^2$, we can read off the induced metric and pullback of the symplectic form as the real and imaginary parts, respectively, of 
\begin{equation*}
	 \sum_s dz^s\otimes d\bar{z}^s (E_k, \bar{E}_l) =    \sum_s r_k r_l \mi \me^{ \mi \theta^k}  \delta_{sk} ( - \mi \me^{-\mi \theta^l} \delta_{sl} ) =  r_k^2 \delta_{kl}\notag .
\end{equation*}
Thus $\mathring{\omega}$ vanishes on $\Sigma_r$, and so $\Sigma_r$ is Lagrangian.  The induced metric is given by $\mathring h_{kl}= r_k^2 \delta_{kl}$.  

Let the ambient connection be $\bar{\nabla}$ (the bar does not denote complex conjugation here).  The covariant derivatives of the tangent vector fields of the embedding with respect to $\mathring g$ in complex notation, are
\begin{equation*}
	\bar \nabla_{E_k} E_l = \frac{\partial}{\partial \theta^k}( \mi r_l  \me^{\mi \theta^l} )\frac{\partial}{\partial z^l}= -r_l \delta_{kl} \me^{\mi \theta^l} \frac{\partial}{\partial z^l} =\delta_{kl}J E_l \, .
\end{equation*}
Since $\Sigma_r$ is Lagrangian, we therefore see that the parallel part $( \bar\nabla_{E_k} E_l )^\|$ vanishes.  
We can now compute the symplectic second fundamental form.  That is, 
\begin{align*}
	\mathring B_{klj} &= \mathring \omega(\bar \nabla_{E_k} E_l - (\bar\nabla_{E_k} E_l)^\| , E_j) 
	= \mathring \omega(\bar \nabla_{E_k} E_l, E_j) \\
	&= -\Im \sum_s dz^s\otimes d\bar{z}^s \Big(\bar \nabla_{E_k} E_l, E_j)  \\ 
	&= -\Im \sum_s dz^s\otimes d\bar{z}^s \Big(-r_l \delta_{kl} \me^{\mi \theta^l} \frac{\partial}{\partial z^l} ,  r_j \me^{-\mi \theta^j} \frac{\partial}{\partial \bar{z}^j} \Big) \\
	&=  r_m^2 \delta_{km} \delta_{lm} \delta_{jm}, 
\end{align*} 
where $m$ can be any of $k$, $l$ or $j$. This emphasizes the symmetry of $\mathring B$ in its indices, as proved more generally above.  From here we see $\mathring H_j = \mathring h^{kl} \mathring{B}_{klj}= 1$ for each $j$.
\end{proof}

\begin{rmk} 
Note that the previous line shows that these Hamiltonian stationary Lagrangian tori are \emph{not minimal}. \end{rmk}

Lemma \ref{lemma:torus} suggests that we should choose a point $p \in M$, find local complex coordinates in a neighbourhood $\mathcal V$ of $p$ as in Section \ref{sec:kaehlergeom}, scale these coordinates by a factor $\rho$ as above.  Then if we embed the submanifold $\Sigma_r $ into the coordinate image of $\mathcal V$, then it remains the case that $\torus$ is Hamiltonian stationary Lagrangian with respect to the standard K\"ahler structure but it is no longer necessarily so with respect to the K\"ahler structure $(g, \omega, J)$ with K\"ahler potential $F_\rho$.   However, if $\rho$ is sufficiently small, then $\Sigma_r$ is totally real; moreover, it is close, in a sense that will be made more precise later on, to being Hamiltonian stationary Lagrangian.

\subsection{The Equations to Solve}

An exactly Hamiltonian stationary Lagrangian submanifold with respect to the K\"ahler structure $(g, \omega, J)$ near the submanifold $\torus$ when $\rho$ is sufficiently small will be found by perturbing $\torus$ appropriately.  This will be done by first defining a class of deformations of $\torus$ and then selecting the appropriate deformation by solving a differential equation.  Define these deformations as follows.  For every function $X : \T^2 \rightarrow \R^2$ of suitably small norm, define an embedding $\mu_X : \T^2 \hookrightarrow \C^2$ by 
\begin{equation*}
	\mu_X : (\theta^1, \theta^2)  \longmapsto \big( r_1(1+ X^1(\theta)) \me^{\mi \theta^1} , r_2 (1+ X^2(\theta)) \me^{\mi \theta^2} \big) \, .
\end{equation*}
Note that the Euclidean-normal bundle of $\torus$ coincides with the bundle $J (T \torus)$ and is spanned by the Euclidean-orthonormal vector fields $N_k :=  \me^{\mi \theta^k} \frac{\partial}{\partial z^k} $ for $k = 1, 2$.  Thus a geometric interpretation of this embedding is to view $X$ as a section of the bundle $J (T \torus)$ and $\mu_X$ as the Euclidean-exponential map scaled by the radii $r_1, r_2$ in the different coordinate directions.   We employ the slight abuse of notation $\mu_X (\torus):=\mu_X(\T^2)$.

Finding $X \in \Gamma( J (T\torus))$ so that $\mu_X (\torus)$ is Hamiltonian stationary Lagrangian with respect to the K\"ahler structure $(g, \omega, J)$ amounts to solving two equations:
\begin{equation}
	\label{eqn:pde}
	\begin{gathered}
		\mu_X^\ast \, \omega = 0 \\
		\nabla \cdot H (\mu_X(\torus)) = 0 
	\end{gathered}
\end{equation}
where $H(\torus)$ is the symplectic mean curvature of $\torus$.  Thus one should consider the differential operator $\Phi_\rho : \Gamma(J(T \torus)) \rightarrow \Lambda^2(\torus) \times \Lambda^0(\torus)$ given by
\begin{equation*}
	\Phi_\rho (X) := \big( \mu_X^\ast \, \omega, \nabla \cdot H (\mu_X(\torus) \big)
\end{equation*}
and attempt to solve the equation $\Phi_\rho(X) = (0,0)$.  Note that the first of these equations is first-order in the vector field $X$ while the second equation is third-order in $X$.   Since $\torus$ is generally \emph{not} Hamiltonian stationary nor Lagrangian with respect to the K\"ahler structure $(g, \omega, J)$ when $\rho > 0$, then $\Phi_\rho(0)$ is a non-vanishing tensor field on $\torus$ depending continuously on $\rho$ in some way that will be determined in the sequel.  Certainly, however, one can assert that $\Phi_{0} (0) = (0,0)$. 

It turns out that, as it stands, equation \eqref{eqn:pde} does not represent a strictly elliptic problem.  A few refinements are necessary in order to achieve this.  First, an important observation to make is that the operator $\Phi_\rho$ maps onto a much smaller space.  In fact, it is true that the first component of $\Phi_\rho(X)$ belongs to $\dif \Lambda^1(\torus)$, the set of exact one-forms, which can be seen as follows. Observe that $\mu_X^\ast \, \omega$  is closed and belongs to the same cohomology class as $\mu_{t X}^\ast \, \omega$ for all $t \in [0,1]$.  But $\mu_0^\ast \, \omega = \dif \alpha \big|_{\Sigma_r}$ where $\alpha$ is the Liouville form, so that $\mu_0^\ast \, \omega$ is exact.  Therefore  $\mu_X^\ast \, \omega$ is exact as well.  The second factor of $\Phi_\rho(X)$ is a divergence; hence its integral against the volume form of $\mu_X(\Sigma_r)$ must vanish. 

Next, we make an \emph{Ansatz} for the section $X$ of the bundle $J(T\torus)$.  We write $X:=  X^k J E_k$ where $E_k := \mi r_k \me^{\mi \theta^k } \frac{\partial}{\partial z^k}$ are the coordinate basis vectors of the tangent space $T\Sigma_r$, and motivated by the Hodge decomposition, we split $X$ into a gradient and a curl component with respect to the metric induced on $\Sigma_r$ by the Euclidean ambient metric.  More specifically, we choose $X := \mathcal X(u, v)$ so that $X \elbow \omega \big|_{\torus} = \dif v + \mathring \star \, \dif u$ for functions $u, v: \torus \rightarrow \R$, where $\mathring \star$ is the Hodge star operator of $\torus$ with respect to the Euclidean metric.  By inspection, this outcome is achieved by the vector field\begin{equation}
	\label{eqn:ansatz}
	\mathcal X(u, v)  :=  \sum_{k} \frac{1}{r_k^2}  \Bigg( \frac{\partial v}{\partial \theta^k} + \sum_j \eps_{k}^j \frac{\partial u}{\partial \theta^j} \Bigg) r_k \me^{\mi \theta^k} \frac{\partial}{\partial z^k}
\end{equation}
where $\eps_{k}^j$ satisfies $\eps_{1}^1 = \eps_{2}^2 = 0$ and $\eps_{1}^2 = -r_1/ r_2$ and $ \eps_{2}^1 = r_2/r_1$.   Note that the mapping given by $(u, v) \mapsto \mathcal X(u, v)$ is linear in $(u, v)$ and independent of $\rho$

Using the \emph{Ansatz} above, one can re-formulate \eqref{eqn:pde} as a pair of equations for the functions $u$ and $v$ which will turn out to be elliptic.  Since \eqref{eqn:pde} is mixed a first- and third-order partial differential equation and $\mathcal X(u, v)$ takes one additional derivative, the functions $u$ and $v$ will be assumed to lie in $C^{4, \alpha}$.  Moreover, since $\mathcal X(u, v)$ clearly remains unchanged if a constant is added to either $u$ or $v$, we impose the normalization 
$$\int_{\torus} u \, \dif \mathrm{Vol}_{\Sigma_r}^\circ = \int_{\torus} v \, \dif \mathrm{Vol}_{\Sigma_r}^\circ  = 0$$ 
where $\dif \mathrm{Vol}_{\Sigma_r}^\circ $ is the volume form of $\Sigma_r$ with respect to the metric induced on $\Sigma_r$ by the ambient Euclidean metric.   Therefore define a new differential operator by 
\begin{gather*}
	\bPhi_\rho : C_0^{4,\alpha}(\torus) \times C_0^{4, \alpha}(\torus) \rightarrow  C^{2,\alpha}(\dif  \Lambda^1(\torus)) \times C^{0, \alpha} (\torus) \\
	 \bPhi_\rho(u, v) := \Phi_\rho  \circ \mathcal X(u, v) \, .
\end{gather*}
where we use the zero subscript to denote a function space upon which our normalization has been imposed.

\section{Analysis of the Hamiltonian Stationary Lagrangian Operator}
\label{sec:anop}

In order to solve the equation $\bPhi_\rho(u, v) = (0,0)$ perturbatively, it is necessary to understand the mapping properties of the linearization of the operator $\bPhi_\rho$ at $(0,0)$.  We will use the notation $\boldsymbol{L}_\rho := \Dif_{(0,0)} \bPhi_\rho$ as well as $L_\rho := \Dif_0 \Phi_\rho$ in the remainder of the paper.  Observe that $\boldsymbol{L}_\rho = L_\rho \circ \mathcal X$ by linearity.   Furthermore, since $\bPhi_\rho$ for $\rho >0$ will often be compared with its Euclidean analogue at $\rho = 0$, we introduce the notation $\mathring \bPhi := \bPhi_{0}$ and $\mathring \Phi := \Phi_{0}$ in keeping with the convention of adorning objects associated with the Euclidean metric with ``$\circ$".  Thus we shall denote the linearizations of these operators by $\mathring L := \Dif_{0} \mathring \Phi$ and $\mathring{\boldsymbol{L}} := \mathring L \circ \mathcal X$, respectively.  Again, note that $\mathring{\boldsymbol{L}} = \mathring L \circ \mathcal X$.

This section contains the following material.  First we compute linearized operator $\mathring{\boldsymbol{L}}$ and determine its kernel.  It will turn out that $\mathring{\boldsymbol{L}} $ is not self-adjoint; hence we next compute the adjoint $\mathring{\boldsymbol{L}}^\ast$ and compute its kernel.   Finally, we compute $\boldsymbol{L}_\rho$ with enough detail to be able to give estimates, in terms of $\rho$, for the difference $\boldsymbol{P}_\rho := \boldsymbol{L}_\rho - \mathring{\boldsymbol{L}}$.




\subsection{The Unperturbed Linearization}
\label{sec:unpertlin}

Let $\mathring \bPhi$ be the Hamiltonian stationary Lagrangian differential operator with respect to the standard K\"ahler structure $(\mathring g, \mathring \omega, J)$.  The task at hand is to compute its linearization at zero, denoted by $\mathring{\boldsymbol{L}}$.  Since $\mathring \bPhi = \mathring \Phi \circ \mathcal X$ and $\mathcal X$ is linear, the main computation is to find the linearization at zero of $\mathring \Phi$ acting on sections of $J(T \torus)$, denoted by $\mathring L$.  In the computations below, repeated indices are summed, a comma denotes ordinary differentiation and a semi-colon denotes covariant differentiation.

\begin{prop}
	\label{prop:unpertlin}
	Let $\Sigma\subset \C^2$ be Lagrangian for the standard symplectic structure.  Let $X$ be a $C^{3}$ section of $N(\Sigma)= J(T \Sigma)$, and write $X := X^j J E_j$ where $E_1, E_2$ is a coordinate basis for the tangent space of $\Sigma$.  Write $\mathring L(X) := \big( \mathring L^{(1)}(X), \mathring L^{(2)}(X) \big)$.  Then
	\begin{align*}
		\mathring L^{(1)}(X) &:= \dif \big( X \elbow \mathring \omega \big) \\
		\mathring L^{(2)}(X) &=  - (\mathring\Delta X^m )_{;m} -  \mathring h^{lm}\mathring h^{sk} \mathring H_s \big(  X^u \mathring B_{lku} \big)_{;m} + \mathring h^{km} \mathring H_k   \big(X^u \mathring H_{u} \big)_{;m}  - \mathring h^{lm} \mathring h^{js} \mathring h^{kq} \big( X^u \mathring B_{squ} \mathring B_{jkl} \big)_{;m} \, .
	\end{align*}	
	\end{prop}	
	
\begin{proof} 
The formula for $\mathring L^{(1)}$ is straightforward.  Recall that it is a standard computation involving the Lie derivative of a 2-form to show that $\left. \frac{\dif}{\dif t} \mu_{tX}^\ast \mathring \omega \right|_{t=0} = \dif (X \elbow \mathring \omega) + X \elbow \dif \mathring \omega$.  Therefore since $\dif \mathring \omega = 0$ then $\mathring L^{(1)}(X) =  \dif \big( X \elbow \mathring \omega \big)$ as desired.  

The remainder of the proof concentrates on the computation for $\mathring L^{(2)}(X)$.  Let $\Sigma$ be a Lagrangian submanifold of $\C^2$ carrying the Euclidean metric $\mathring g$, and let $X$ be a section of the normal bundle of $\Sigma$.  Let $\mu_t : \C^2 \rightarrow \C^2$ be a one-parameter family of diffeomorphisms with $\frac{\dif}{\dif t} \mu_t \big|_{t=0} = X$ and set $\Sigma^t := \mu_t(\Sigma)$.   Next, choose $E_1,  E_2$ a local coordinate frame for $\Sigma$ coming from geodesic normal coordinates at $p_0\in \Sigma$ in the induced metric $\mathring h$ at $t=0$.  Then $JE_1, JE_2$ is basis for the normal bundle of $\Sigma$ at $t=0$, because $\Sigma$ is Lagrangian, but it does not necessarily hold for $|t|\neq 0$ since $\mu_t$ is not assumed to be a family of symplectomorphisms.   However, for $p$ near $p_0$, and $T_pM=T_p\Sigma\oplus J(T_p\Sigma)$.  We write $X\big|_{\Sigma} = X^j JE_j$.  Note that $X$ and $E_k$ commute along $\mu_t$, and since $X$ is transverse to $\Sigma$, we can extend the fields $E_k$ locally using the diffeomorphism $\mu_t$ to a basis for $T_{\mu_t(p)} \Sigma^t$, for $|t|$ small.  In these coordinates the matrix for $\mathring h$ on $\Sigma^t$ is the same as that for $\mu_t^*\mathring g$ on $T\Sigma$.   The computations below are evaluated at $p_0$ at $t=0$.  

In terms of the local coordinates introduced above, we have
$$ \mathring \nabla \cdot \mathring H (\Sigma^t) = \mathring h^{lm} \mathring h^{jk} \mathring B_{jkl;m}$$
where $\mathring h_{kl} := \mathring g (E_k, E_l)$ is the induced metric,  $\mathring h^{jk}$ are the components of the inverse of the induced metric, $\mathring\nabla$ is the induced connection, and
\begin{align}
	\label{eqn:divmeancurvcomp}
	\mathring B_{jkl} := \mathring\omega ( (\mathring \nabla_{E_j} E_k )^\perp, E_l) = \mathring\omega (\bar \nabla_{E_j} E_k , E_l) - \mathring\Gamma_{jk}^s \mathring \omega( E_s, E_l )
\end{align}
with $\mathring \Gamma_{jk}^s$ the Christoffel symbols of $\mathring h_{jk}$ and $\bar \nabla$ the ambient Euclidean connection.

The terms in \eqref{eqn:divmeancurvcomp} all depend on $t$.  Since $\mathring\nabla \cdot \mathring H (\Sigma^t) = \mathring h^{lm}  \mathring H_{l;m} = \mathring h^{lm} \mathring H_{l,m} - \mathring h^{lm} \mathring \Gamma_{lm}^s \mathring H_{s} $ where $\mathring H_l := \mathring h^{jk} \mathring B_{jkl}$, differentiating \eqref{eqn:divmeancurvcomp} at $t=0$ yields
$$\left. \frac{\dif}{\dif t} \mathring\nabla \cdot \mathring H (\Sigma^t) \right|_{t=0} = ( \mathring h^{lm})' \mathring H_{l;m}  - \mathring h^{lm}  ( \mathring\Gamma_{lm}^s)'  \mathring H_{s} + \mathring h^{lm} \big( ( \mathring H_{l})'  \big)_{;m}$$
where a prime denotes the value of the time derivative at zero. 

Expressions for $( \mathring h^{lm})'$ and $(\mathring \Gamma_{lm}^s)'$ and $(\mathring H_{l})' $ are now required.  To begin, it is straightforward to compute
\begin{align*}
	(\mathring h^{lm} )' &= - 2 \mathring h^{ls} \mathring h^{mq} X^u \mathring B_{squ} \\
	(\mathring \Gamma_{lm}^s)' &= \mathring h^{sq} \Big( \big(X^u \mathring B_{lqu} \big)_{;m} + \big(X^u \mathring B_{mqu} \big)_{;l} -  \big(X^u \mathring B_{lmu} \big)_{;q}  \Big)  \, .
\end{align*}
Next
\begin{equation*}
	\big( \mathring H_{l} \big)' = (\mathring h^{jk})' \mathring B_{jkl} + \mathring h^{jk} (\mathring B_{jkl} )' =  - 2 \mathring h^{js} \mathring h^{kq} X^u \mathring B_{squ} \mathring B_{jkl} + \mathring h^{jk} (\mathring B_{jkl} )'
\end{equation*}
and the fact that both $\mathring\Gamma_{jk}^s(p_0)$ and $\mathring \omega \big|_{\Sigma^t}$ vanish at $t=0$ implies 
\begin{align*}
	(\mathring B_{jkl} )' &= \left. \frac{\dif}{\dif t} \Big(\mathring \omega \big(\bar \nabla_{E_j} E_k , E_l \big) - \mathring\Gamma_{jk}^s \mathring\omega ( E_s, E_l )
 \Big) \right|_{t=0} \\
	&= \left. \frac{\dif}{\dif t} \Big( \mathring\omega \big( \bar \nabla_{E_j} E_k , E_l \big)  \Big) \right|_{t=0} \\
	&= \mathring \omega \big( \bar\nabla_X \bar\nabla_{E_j} E_k , E_l \big) + \mathring\omega \big(\bar\nabla _{E_j} E_k , \bar\nabla_X E_l \big) \\
	&= \mathring\omega \big(\bar \nabla_{E_j} \bar\nabla_{E_k} X, E_l  \big) + \mathring\omega \big(\bar\nabla _{E_j} E_k , \bar \nabla_{ E_l} X \big) \\
	&=E_j  \mathring\omega \big(\bar\nabla_{E_k} X, E_l  \big) - \mathring\omega \big(\bar\nabla_{E_k} X, \bar \nabla_{E_j} E_l  \big) + \mathring\omega \big(\bar\nabla _{E_j} E_k , \bar \nabla_{ E_l} X \big) \\
	&= - E_j  \mathring g \big( \bar\nabla_{E_k} (X^q E_q), E_l  \big) + \mathring g  \big(  \bar\nabla_{E_k} (X^q E_q), \bar \nabla_{E_j} E_l  \big) + \mathring g \big( \bar\nabla _{E_j} E_k , \bar \nabla_{ E_l} (X^q E_q) \big)
\end{align*}
where $\bar \nabla$ is the ambient connection; we have used that $X$ commutes with $E_k$ along $\mu_t$, that the ambient curvature vanishes, and that $\mathring\omega$ is parallel.  Now 
$$\bar\nabla_{ E_l} (X^q E_q ) = X^q_{,l} E_q + X^q  \bar \nabla_{E_l} E_q =  X^q_{;l} E_t  - X^q \mathring h^{uv} \mathring B_{lqu} JE_v \, .$$  
Note that at $t=0$, $\bar \nabla_{E_k} E_j$ is normal to $\Sigma$ at $p_0$, and moreover $\mathring g(\bar\nabla _{E_j} E_k, J E_m)= -\mathring B_{jkm}$ at $p_0$.  Thus we have 
\begin{align*}
	(\mathring B_{jkl})' &= - E_j  \mathring g \big( X^q_{;k} E_q - X^q \mathring h^{uv} \mathring B_{kqu} JE_v, E_l  \big) + \mathring g  \big(X^q_{;k} E_q  - X^q \mathring h^{uv} \mathring B_{kqu} JE_v , \bar \nabla_{E_j} E_l  \big) \\
	&\qquad + \mathring g \big( \bar\nabla _{E_j} E_k , X^q_{;l} E_q  - X^q \mathring h^{uv} \mathring B_{lqu} JE_v \big) \\ 
	&= - X^q_{;kj} \mathring h_{ql}  + X^q \mathring B_{kqu} \mathring B_{jlv} \mathring h^{uv}  + X^q \mathring B_{jku} \mathring B_{lqv} \mathring h^{uv}  \, .
\end{align*}
Everything can now be put together:
\begin{align*}
	\mathring L^{(2)}(X)  &= - 2 \mathring h^{ls} \mathring h^{mq} X^u \mathring B_{squ} \mathring H_{l;m} \\[-1ex]
	&\qquad -  \mathring h^{sq} \mathring H_s \big( 2 \big(X^u \mathring B_{lqu} \big)_{;m}\mathring h^{lm}  -  \big(X^u \mathring H_{u} \big)_{;q}  \big)   \\
	&\qquad  - 2 \mathring h^{lm} \mathring h^{js} \mathring h^{kq} \big( X^u \mathring B_{squ} \mathring B_{jkl} \big)_{;m} \\
	&\qquad + \mathring h^{lm} \mathring h^{jk}\big( - X^q_{;kj} \mathring h_{ql} + X^q \mathring B_{kqu} \mathring B_{jlv} \mathring h^{uv}  + X^q \mathring B_{jku} \mathring B_{lqv} \mathring h^{uv} \big)_{;m} \\[1ex]
	&= - (\mathring\Delta X^m )_{;m}-  \mathring h^{sq} \mathring H_s \big(  \big(X^u \mathring B_{lqu} \big)_{;m}\mathring h^{lm}  -  \big(X^u \mathring H_{u} \big)_{;q}  \big) \\
	&\qquad  - \mathring h^{lm} \mathring h^{js} \mathring h^{kq} \big( X^u \mathring B_{squ} \mathring B_{jkl} \big)_{;m}
\end{align*}
This is the desired formula.  
\end{proof}

To compute $\mathring L^{(2)}$ for the torus $\Sigma_r$, note that both $\mathring B$ and $\mathring H$ are parallel tensors in this case.  Consequently the second fundamental form term in $\mathring L^{(2)}$ becomes simply $X \mapsto -\mathring A^{l}_k X^k_{;l}$ where 
$$\mathring A^{kl} := \mathring H_s \mathring B^{lsm} - \mathring H^l \mathring H^m+ \mathring B_{sq}^m\mathring B^{sqm}$$
and furthermore, we can compute precisely: substituting and $\mathring h_{kl} = r_k^2 \delta_{kl} $ and $\mathring B_{jkl} = r_s^2 \delta_{sj} \delta_{sk} \delta_{sl}$ for the induced metric and symplectic second fundamental form of $\torus$ with respect to the Euclidean metric yields 
$$\mathring A^{lm}= \frac{ 2 \delta^{lm} }{r_m^4}- \frac{1}{r_l^{ 2} r_m^{2} } \, .$$
Now let $X = \mathcal X (u, v)$ as in \eqref{eqn:ansatz} and substitute this into the formul\ae\ of Proposition \ref{prop:unpertlin} to find the linearization $\mathring{\boldsymbol{L}}$.

\begin{cor}
	Let $(u, v) \in C_0^{4, \alpha}(\torus) \times C_0^{4, \alpha}(\torus)$.  Write $\mathring{\boldsymbol{L}} = \big( \mathring{\boldsymbol{L}}^{(1)}, \mathring{\boldsymbol{L}}^{(2)} \big)$.  Then 
	\begin{align*}
		\mathring{\boldsymbol{L}}^{(1)}(u, v) &:= \dif \, \mathring \star \, \dif u \\
		\mathring{\boldsymbol{L}}^{(2)}(u, v) &:=  \mathring \Delta ( \mathring \Delta v ) + \mathring A^{lm} v_{;lm}  + \mathring A^{lm} \eps_{l}^k u_{;mk}\, .
	\end{align*}
\end{cor}

\noindent

\subsection{The Kernel of the Unperturbed Linearization}
\label{sec:kernel}

The determination of the kernel of the linearized operator $\mathring{\boldsymbol{L}}$ is best done in two stages.  First one finds the kernel of $\mathring L$ and then one takes into account the effect of $\mathcal X$.  Thus the starting point is to express the formul\ae\ of Proposition \ref{prop:unpertlin} explicitly in local coodinates.  To this end, suppose that $\torus$ is given in local coordinates by its standard embedding.  Make the \emph{Ansatz} $X:= \sum_k X^k (-r_k\me^{\mi \theta^k} \frac{\partial}{\partial z^k})$ for the deformation vector field in the formul\ae\ from Proposition \ref{prop:unpertlin} to obtain
$$ \mathring L (X) = -\left(  \sum_{i,k} r_k^2 X^k_{,i} \dif \theta^i \wedge \dif \theta^k \, \, , \, \,   \sum_{i,k}  \frac{1}{r_k^{2}} \big( X^i_{,kki} -  X^i_{,k} \big) + \sum_{i} \frac{2}{r_i^2} X^i_{,i} \right) \, .$$
The operator $\mathring L$ thus becomes a constant-coefficient differential operator on the torus.  Solving the equation $\mathring L (X) = (0,0)$ for the kernel of $\mathring L$ thus becomes a matter of Fourier analysis.  (Note: this calculation appears in \cite{oh2} for the $n$-dimensional torus; it is included here for the sake of completeness.)

\begin{prop}
	Expressed in the local coordinates for the standard embedding of $\torus$, the kernel of $\mathring L$ consists of vector fields $X :=  \sum_k X^k (-r_k\me^{\mi \theta^k} \frac{\partial}{\partial z^k})$ where
	$$X^k = \lambda_k +  \frac{1}{r_k^2} \frac{\partial f}{\partial \theta^k}$$
	with $f(\theta) := a +  \sum_j \big( b_{j1} \cos(\theta^j) + b_{j2} \sin(\theta^j) \big) + c_1 \cos(\theta^1 - \theta^2) + c_2 \sin(\theta^1 - \theta^2)$ and $a, b_{js}, c_s, \lambda_k \in \R$.  
\end{prop}

\begin{proof}  
The first equation in $\mathring L (X) = (0,0)$ implies either: that $X^k$ is constant for every $k$, and thus the one-form $r_k^2 X^k \dif \theta^k$ is harmonic on $\Sigma_r$; or else that there is a function $f : \T^2 \rightarrow \R$ with 
$$ X^k = \frac{1}{r_k^2} \frac{\partial f}{\partial \theta^k} \, .$$

In the first case, the second equation in $\mathring L (X) = (0,0)$ is satisfied trivially.  Note that a one-form of this type is \emph{not} exact, implying that $X$ is \emph{not} induced by a Hamiltonian vector field.  In the second case, insert $X^k  := r_k^{-2} \frac{\partial f}{\partial \theta^k}$ into the second equation to find
\begin{align*}
	\sum_{i,k} \frac{1}{r_i^2 r_k^2} \big(  f_{,iikk} - f_{,ik} \big) + \sum_i \frac{2}{r_i^4} f_{,ii} = 0 \, .
\end{align*}
This is a constant-coefficient, fourth order elliptic equation on the  torus which can be solved by taking the discrete Fourier transform. The Fourier coefficients $\hat f (\vec n) := \langle f, \me^{\mi \vec \theta \cdot \vec n} \rangle $ of the solutions must thus satisfy
$$ \left(  \sum_{i,k} \frac{n_i^2 n_k^2 + n_i n_k }{r_i^2 r_k^2} - \sum_i \frac{2 n_i^2}{r_i^4} \right) \hat f(\vec n) = 0 \, .
$$
The trivial solution of this equation is $n_1 = n_2 = 0$ and this corresponds to the constant functions.  There are also non-trivial solutions of this equation: either  $n_i =  \pm 1$ for some fixed $i$ and all other $n_k =0$; or else $n_i = \pm 1$ and $n_j = \mp 1$ for $i\neq j$.   The fact that there are no other non-trivial solutions can be seen as follows.  Summing over $i, k \in \{1, 2 \}$ explicitly and re-arranging terms yields the equation $n_1^2 \pm  n_1 + r_1^2 r_2^{-2} \big( n_2^2 + n_2 \big) = 0$.  But since the quadratic $x^2 \pm x + C^2$ only has the integer roots $x = 0,1$ when $C=0$ and no integer roots when $C \neq 0$, it must be the case that $(n_1, n_2) = (1,0), (0,1), (1,-1)$ or $(-1,1)$.   Computing the inverse Fourier transform now yields the desired vector fields in the kernel of $\mathring L$.
\end{proof}

Observe that there is a geometric interpretation of the kernel of $\mathring L$.  The one-parameter families of complex structure-preserving isometries of $\C^2$ are the unitary rotations and the translations.  Each of these is a Hamiltonian deformation where the Hamiltonians are given by linear functions in the first case and quadratic polynomials of the form $z \mapsto z^\ast \cdot A \cdot z$ in the second case, where $A$ is a Hermitian matrix.  Of these, only the non-diagonal Hermitian matrices generate non-trivial motions of  $\torus$.  The restrictions of these Hamiltonian functions to $\torus$ are the functions of the form
\begin{equation}
	\label{eqn:kernel}
	f(\theta) =  \sum_j \big( b_{j1} \cos(\theta^j) + b_{j2} \sin(\theta^j) \big) + c_1 \cos(\theta^1 - \theta^2) + c_2 \sin(\theta^1 - \theta^2) \qquad a_k, b_{js}, c_s \in \R
\end{equation}
in the kernel of $\mathring L$.  The remaining elements of the kernel of $\mathring L$ derive from another set of deformations of $\torus$ which preserve both the Lagrangian condition and the Hamiltonian-stationarity.  These arise from allowing the radii of $\torus$ to vary --- in other words the deformations $\Sigma^t := \Sigma_{r + a t }$ for some $a = (a_1, a_2)$.  

The effect of the substitution $X = \mathcal X(u, v)$ is to restrict to a space of deformations that are transverse to those deformations for which $X \elbow \mathring \omega$ is closed but non-exact.  In particular, this excludes the harmonic one-forms from the kernel of the operator $\mathring{\boldsymbol{L}}$.

\begin{cor}
	\label{cor:kernel}
	The kernel of $\mathring{\boldsymbol{L}}$ is
	$$\mathcal K :=  \{ 0 \} \times \mathrm{span}_{\R} \{ \cos(\theta^1), \cos(\theta^2),  \sin(\theta^1), \sin(\theta^2), \cos(\theta^1 - \theta^2), \sin(\theta^1 - \theta^2) \} \, .$$ 
\end{cor}

\noindent Note: the constant functions are not in $\mathcal K$ because the conditions $\int_{\torus} u \, \dif \mathrm{Vol}_{\Sigma_r}^\circ = \int_{\torus} v \, \dif \mathrm{Vol}_{\Sigma_r}^\circ = 0$ have been imposed on functions in the domain of $\mathring{\boldsymbol{L}}$.

\subsection{The Adjoint of the Unperturbed Linearization}
\label{sec:adjoint}

The operator $\mathring{\boldsymbol{L}}$ computed in Section \ref{sec:unpertlin} is not self-adjoint.  Thus it is necessary to compute its adjoint and find the kernel of its adjoint in order to determine a space onto which $\mathring{\boldsymbol{L}}$ is surjective. 

\begin{prop} 
	The formal $L^2$ adjoint of $\mathring{\boldsymbol{L}} : C_0^{4,\alpha}(\torus) \times C_0^{4, \alpha}(\torus) \rightarrow  C^{2,\alpha}(\dif  \Lambda^1(\torus)) \times C^{0, \alpha} (\torus)$ is the operator $\mathring{\boldsymbol{L}}^\ast := \big( [\mathring{\boldsymbol{L}}^\ast ]^{(1)}, [ \mathring{\boldsymbol{L}}^\ast ]^{(2)} \big) : C^{4,\alpha}(\dif  \Lambda^1(\torus)) \times C^{4, \alpha} (\torus) \rightarrow C_0^{2,\alpha}(\torus) \times C_0^{0, \alpha}(\torus)$ where 
	\begin{equation}
		\label{eqn:adjoint}
		\begin{aligned}[]
			 [\mathring{\boldsymbol{L}}^\ast ]^{(1)} (u, v) &:= \mathring \Delta u + \mathring A^{ml} \eps_{l}^k v_{;mk} \\
			[ \mathring{\boldsymbol{L}}^\ast ]^{(2)} (u, v) &:=  \mathring \Delta ( \mathring \Delta v ) + \mathring A^{lm} v_{;lm} \, .
		\end{aligned}
	\end{equation}
	and $\mathring A^{lm}= 2r_m^{-4} \delta^{lm}- r_l^{-2}r_m^{-2}$ as computed earlier.
\end{prop}
\begin{proof}
	Straightforward integration by parts based on the formul\ae\ for $\mathring{\boldsymbol{L}}$ and $\mathcal X $.
\end{proof}

The kernel $\mathcal K^\ast $ of the adjoint $\mathring{\boldsymbol{L}}^\ast$ is now easy to find, given the formula \eqref{eqn:adjoint}.  Consider the equation $\mathring{\boldsymbol{L}}^\ast (\star u, v) = (0,0)$ for $(u, v) \in C^{4,\alpha}_0 (\torus) \times C^{4, \alpha} (\torus)$.  The second of these equations along with the calculations of Section \ref{sec:kernel} implies that $v$ is of the form \eqref{eqn:kernel} found before.  Now $u$ can be determined from the first of these equations via $\mathring \Delta u = -\mathring A^{lm} \eps_{l}^k v_{;mk}$.  Since the form of $\mathring A^{lm} $ is known, one can in fact determine $u$ explicitly.  Note that we will employ a slight abuse of notation below by identifying $C^{k,\alpha}_0 (\torus)$ with $C^{k,\alpha}(\dif  \Lambda^1(\torus))$ via the Hodge star operator.

\begin{cor}  
	\label{cor:cokernel}
	The kernel of $\mathring{\boldsymbol{L}}^\ast$ is
\begin{align*}
	\mathcal K^\ast  &:=  \mathrm{span}_{\R} \{ (0,1) \} \oplus  \mathrm{span}_{\R} \big\{   \cos(\theta^1) \cdot  (1, r_1 r_2 )  \, , \,  \cos(\theta^2) \cdot  (1, - r_1 r_2 )  \,  , \\
	&\qquad \qquad \qquad \qquad \qquad \qquad \quad \sin(\theta^1) \cdot  (1, r_1 r_2 ) \, , \, \sin(\theta^2) \cdot  (1, - r_1 r_2 ) \, , \\
	&\qquad \qquad \qquad \qquad \qquad \qquad \qquad \quad \cos(\theta^1 - \theta^2)\cdot (0,1) \, , \,  \sin(\theta^1 - \theta^2) \cdot (0,1) \big\}   \, .
\end{align*}
\end{cor}

\noindent Note that the projections of the  $\cos(\theta^1- \theta^2)$ and $\sin(\theta^1- \theta^2)$ co-kernel elements to the first coordinate vanish; this fact will be used crucially later on.

\subsection{The Perturbed Linearization}
\label{sec:perturbterm}

Let $\bPhi_{\rho}$ be the Hamiltonian stationary Lagrangian differential operator with respect to the K\"ahler structure $(g, \omega, J)$ corresponding to the K\"ahler potential $F_{\rho}(z, \bar z) = \frac{1}{2} \| z \|^2 + \rho^2 \hat F_\rho( z, \bar z)$ with $\rho > 0$.  The task at hand is to compute its linearization at zero, denoted by $\boldsymbol{L}_\rho$ and express it as a perturbation of $\mathring{\boldsymbol{L}}$ in the form $\boldsymbol{L}_\rho = \mathring{\boldsymbol{L}} + \boldsymbol{P}_\rho$.  Then the dependence of   $\boldsymbol{P}_\rho $ on $\rho$ must be analyzed.  Since $\bPhi_\rho = \Phi_\rho \circ \mathcal X$ and $\mathcal X$ is linear, once again it is best to start with the linearization of $\Phi_\rho$ acting on sections of $J(T \torus)$, denoted by $L_\rho$.  In the computations below, repeated indices are summed, a comma denotes ordinary differentiation and a semi-colon denotes covariant differentiation with respect to the induced metric.

\begin{prop}
	\label{prop:pertlin}
	Let $\Sigma$ be a totally real submanifold of $\C^2$ equipped with the K\"ahler metric $g$.  Let $X$ be a $C^3$ section of $J(T \Sigma)$ and write $X := X^j J E_j$ where $\{ E_1, E_2\}$ is a coordinate basis for the tangent space of $\Sigma$.  Write $L_\rho(X) := \big( L_\rho^{(1)}(X), L_\rho^{(2)}(X) \big)$.  Then
	\begin{align*}
		L_\rho^{(1)}(X) &:= \dif \big( X \elbow \omega \big) \\
		L_\rho^{(2)}(X) &:=  \mathcal E_1 (X) +\mathcal{E}_2(X) 
\end{align*}
where 
\begin{align*}
	\mathcal E_1(X) &:= - (\Delta X^m )_{;m}  - h^{lm}  X^s \bar R_{sl} - h^{lm}h^{qu} H_{q;m}X^s B_{usl}  \\
	&\qquad  + h^{lm}h^{jk}h^{qu}\big( X^s ( B_{ksq}B_{jlu} - B_{ksq} B_{jul} - B_{qsk}B_{jul}) \big)_{;m} \\
	&\qquad -h^{lm}h^{qu}H_{u}\big( X^s B_{qsl}\big)_{;m}+h^{lm}h^{qu}H_{u}\big( X^s B_{lsm}\big)_{;q} \\
	\mathcal{E}_2 (X) &:=  -  h^{lu} h^{qm}(h^{jk} B_{jkl})_{;m} \mathcal C(X)_{uq} - \big( h^{lm} h^{ju} h^{qk} \mathcal C(X)_{uq} B_{jkl}\big)_{;m} \\
	& \qquad - \tfrac{1}{2} h^{lm} h^{jk} h^{sq} B_{jks}  \big( \mathcal C(X)_{ql;m} + \mathcal C(X)_{rm;l} - \mathcal C(X)_{lm;q} \big) \\
	& \qquad + h^{lm} h^{jk} X^s \big(  g  \big(\mathcal{D}((\bar\nabla_{E_k}E_s)^{\perp})  ,  (\bar\nabla_{E_j} E_l )^{\perp} \big) + g \big((\bar\nabla _{E_j} E_k)^{\perp} , \mathcal{D}((\bar\nabla_{E_l}E_s)^{\perp} ) \big)\big)_{;m} \\
	&\qquad - \frac{1}{2}h^{lm}h^{jk}  h^{sq}\omega_{sl} \big( \beta(X)_{qj;k} + \beta(X)_{qk;j} -  \beta(X)_{jk;q} +\mathcal C(X)_{qj;k}+\mathcal C(X)_{qk;j}-\mathcal C(X)_{jk;q}   \big)_{;m} 
\end{align*} 
and also $\mathcal C(X)_{kl} := X^s_{;k}\omega_{sl}+X^s_{;l}\omega_{sk}$, $\beta(X)_{kl} := X^s \big( B_{ksl} + B_{lsk} \big)$, and $\mathcal{D} :T M \rightarrow TM$ is the operator giving the difference between the orthogonal projection of a vector $W \in T_pM$ onto $N_p\Sigma$ and its orthogonal projection onto $J(T_p\Sigma)$.
\end{prop}

\begin{proof}
The formula for $L_\rho^{(1)}$ follows as before; thus consider $L^{(2)}_\rho(X)$.  In general, let $\Sigma$ be a totally real submanifold of $M$.  Let $X$ be a section of the bundle $J(T\Sigma)$.  Let $\mu_t :M \rightarrow M$ be a one-parameter family of diffeomorphisms with $\frac{\dif}{\dif t} \mu_t \big|_{t=0} = X$ and set $\Sigma^t := \mu_t(\Sigma)$.  Note that although $X$ is always transverse to $\Sigma$, it is not necessarily normal to $\Sigma$ because $\Sigma$ is not necessarily Lagrangian.

Next, choose $E_1,  E_2$ a local coordinate frame for $\Sigma$ coming from geodesic normal coordinates at $p_0\in \Sigma$ in the induced metric $h$ at $t=0$.  Then $JE_1, JE_2$ is basis for $J(T_p \Sigma)$ for $p$ near $p_0$, and $T_p M=T_p \Sigma \oplus J(T_p \Sigma)$ for such $p$. We write $X\big|_{\Sigma} = J(X^jE_j)=X^j JE_j$.  Note that $X$ and $E_k$ commute along $\mu_t$, and since $X$ is transverse to $\Sigma$, we can extend the fields $E_k$ locally using the diffeomorphism $\mu_t$ to a basis for $T_{\mu_t(p)} \Sigma^t$, for $|t|$ small.  In these coordinates the matrix for $h$ on $\Sigma^t$ is the same as that for $\mu_t^*g$ on $T\Sigma$.  The computations below are evaluated at $p_0$ at $t=0$. 

In terms of these coordinates, we have
$$ \nabla\cdot H (\Sigma^t) = h^{lm} h^{jk} B_{jkl;m}$$
where $h_{kl} := g (E_k, E_l)$ is the induced metric, $h^{jk}$ are the components of the inverse of the induced metric, $\nabla$ is the induced connection, and
\begin{align}
	\label{eqn:pertdivmeancurvcomp}
	B_{jkl} := \omega ( (\bar\nabla_{E_j} E_k )^\perp, E_l) = \omega (\bar\nabla_{E_j} E_k , E_l) - \Gamma_{jk}^s \omega( E_s, E_l ),
\end{align}
where $\Gamma_{jk}^s$ are the Christoffel symbols of $h_{jk}$, and $ \bar\nabla$ is the ambient connection of $g$.

The terms in \eqref{eqn:pertdivmeancurvcomp} all depend on $t$.  We will now compute the first derivative of \eqref{eqn:pertdivmeancurvcomp} at $t=0$.  By writing
$$\nabla \cdot H (\Sigma^t) = h^{lm}  H_{l;m} = h^{lm} H_{l,m} - h^{lm}  \Gamma_{lm}^s H_{s} $$
we find 
$$\left. \frac{\dif}{\dif t} (\nabla \cdot H (\Sigma^t) ) \right|_{0}= ( h^{lm} )' H_{l;m}  - h^{lm}  ( \Gamma_{lm}^s )' H_{s} + h^{lm} (( H_{l} )' )_{;m} $$
where once again a prime denotes the value of the time derivative at zero.

We compute the first variation of the metric $h$.  The fact that $\Sigma$ is not assumed to be Lagrangian for $\omega$ influences the outcome of the computation.  We have 
\begin{align*}
	(h_{kl})' &=  g(\bar\nabla_{X} E_k , E_l)+g(E_k, \bar\nabla_{X} E_l) \\
	&= g( \bar\nabla_{E_k} X , E_l)+g( \bar\nabla_{E_l} X, E_k) \\
	&= X^s_{;k} g(JE_s, E_l) + X^s g (J \bar \nabla_{E_k} E_s, E_l) + X^s_{;l} g(JE_s, E_k) + X^s g (J \bar \nabla_{E_l} E_s, E_k) \\
	&= X^s_{;k} \omega_{sl} + X^s_{;l} \omega_{sk} + X^s \big( B_{ksl} + B_{lsk} \big) \, .
\end{align*}
Define $\mathcal C(X)_{kl} := X^s_{;k}\omega_{sl}+X^s_{;l}\omega_{sk}$ and $\beta(X)_{kl} := X^s \big( B_{ksl} + B_{lsk} \big)$.  Note that if $\Sigma$ were Lagrangian with respect to $\omega$ then $\mathcal C(X)$ would vanish identically and $\beta(X)$ would equal $2 X^s B_{kls}$.  It is now straightforward to compute
\begin{align*}
	( h^{kl} )' &= -h^{km} h^{lq}  h_{mq} ' = -  h^{km} h^{lq} \big( \beta(X)_{mq} + \mathcal C(X)_{mq} \big)  \\
	( \Gamma_{lm}^k )' &= \frac{1}{2} h^{kq} \big( \beta(X)_{ql;m} + \beta(X)_{qm;l} -  \beta(X)_{lm;q} +\mathcal C(X)_{ql;m}+\mathcal C(X)_{qm;l}-\mathcal C(X)_{lm;q} \big) \, .
\end{align*}
Next we have
\begin{align*}
	(H_{l})'  &= \left. \frac{\dif}{\dif t} \big( h^{jk} B_{jkl} \big) \right|_{t=0} \\
	&=  -  h^{jm} h^{kq} \big( \beta(X)_{mq} + \mathcal C(X)_{mq} \big) B_{jkl} + h^{jk} (B_{jkl})' \, .
\end{align*}
We now use the facts that $\omega$ and $J$ are parallel, that $X$ and $E_k$ commute along $\mu_t$, and $\Gamma_{jk}^s(p_0)$ vanishes at $t=0$ to deduce
\begin{align*}
	(B_{jkl})' &= \left. \frac{\dif}{\dif t}\omega \big( ( \bar\nabla_{E_j} E_k )^\perp , E_l \big) \right|_{t=0}  \\
	&= \omega \big( \bar\nabla_X \bar\nabla_{E_j} E_k , E_l \big) + \omega \big(\bar\nabla _{E_j} E_k , \bar\nabla_X E_l \big) - ( \Gamma_{jk}^s )' \omega_{sl} \\
	&= \omega \big( \bar\nabla_{E_j} \bar\nabla_{E_k} X, E_l  \big) + \omega \big(\bar\nabla _{E_j} E_k , \bar\nabla_{ E_l} X \big)+\omega(\bar R(E_j,X) E_k,E_l) - ( \Gamma_{jk}^s )' \omega_{sl} \\
	&= - E_j  \big[g \big( \bar\nabla_{E_k} (X^{s} E_s ), E_l  \big) \big]+ g  \big(  \bar\nabla_{E_k} (X^{s}E_s),  \bar\nabla_{E_j} E_l  \big) +  g \big( \bar\nabla _{E_j} E_k ,\bar \nabla_{ E_l} (X^{s} E_s) \big) \\
	&\qquad - X^s \bar R_{jskl} -( \Gamma_{jk}^s )' \omega_{sl} \\
	&= - E_j  \big[g \big( X^s_{;k} E_s +(\bar\nabla_{E_k}( X^{s} E_s  ))^{\perp} , E_l  \big)\big] + g  \big(X^s_{;k} E_s +(\bar\nabla_{E_k}( X^{s} E_s ))^{\perp}  , \bar \nabla_{E_j} E_l  \big) \\
	&\qquad + g \big(\bar\nabla _{E_j} E_k , X^s_{;l} E_s +(\bar\nabla_{E_l}( X^{s} E_s ))^{\perp}  \big)  - X^s \bar R_{jskl} - ( \Gamma_{jk}^s )' \omega_{sl}
\end{align*}
where $\bar R_{jskl}$ are the components of the ambient curvature tensor.  Now using the fact that we've arranged to have $\bar\nabla_{E_j} E_k$ orthogonal to $\Sigma$ at $p_0$ at $t=0$, we can deduce
\begin{align}
	\label{eqn:bprime}
	(B_{jkl})' &= - E_j  \big[g \big( X^s_{;k} E_s  , E_l  \big)\big] + g  \big((\bar\nabla_{E_k}(X^sE_s))^{\perp}  ,  \bar\nabla_{E_j} E_l  \big) \notag \\
	&\qquad + g \big(\bar\nabla _{E_j} E_k , (\bar\nabla_{E_l}(X^{s} E_s ))^{\perp}  \big) - X^s \bar R_{jskl} -  ( \Gamma_{jk}^s )' \omega_{sl} \notag \\
	&\begin{aligned} 
		&= -X_{l;kj} + X^s g  \big((\bar\nabla_{E_k}E_s)^{\perp}  ,  (\bar\nabla_{E_j} E_l)^{\perp}  \big) + X^s g \big( (\bar\nabla _{E_j} E_k)^{\perp} , (\bar\nabla_{E_l}E_s)^{\perp}  \big) \\ 
		&\qquad  - X^s \bar R_{jskl}  -  ( \Gamma_{jk}^s )' \omega_{sl} \, . 
	\end{aligned}
\end{align}
	
To deal with the $(\bar\nabla _{E_j} E_k)^{\perp}$ terms we introduce the operator $\mathcal{D}$ on $T_p M$ which is the difference between the orthogonal projection onto $N_p\Sigma$ and the orthogonal projection onto $J(T_p\Sigma)$.  Now, for any $W\in N_p\Sigma$,  we can write 
$$W=h^{ij}g(W, JE_j)JE_i + \mathcal{D}(W)= -h^{ij}\omega (W, E_j)JE_i+\mathcal{D}(W).$$ 
where we've used the fact that $J$ is an isometry.  Consequently \eqref{eqn:bprime} becomes
\begin{align*}
	(B_{jkl})' &=- X_{l;kj}+X^qh^{uq}B_{ksq} B_{jlu} + X^s h^{uq} B_{lsq}B_{jku}\\
	& \qquad +X^s g  \big(\mathcal{D}((\bar\nabla_{E_k} E_s)^{\perp})  ,  (\bar\nabla_{E_j} E_l )^{\perp} \big) + X^s g \big((\bar\nabla _{E_j} E_k)^{\perp} , \mathcal{D}((\bar\nabla_{E_l}E_s)^{\perp} ) \big)\\
	&\qquad  - X^s \bar R_{jskl} - ( \Gamma_{jk}^s )' \omega_{sl}\, .
\end{align*}
	
We have now computed all the separate constituents of $L^{(2)}_\rho(X)$.  It remains only to put everything together.  We find
\begin{align*}
	L^{(2)}_\rho(X)  &= ( h^{lm})' H_{l;m}  - h^{lm}  (\Gamma_{lm}^s)' H_{s} + h^{lm} \big( ( H_{l})'  \big)_{;m}\\
	&=  - h^{lu} h^{qm} (h^{jk} B_{jkl})_{;m} \big (\beta(X)_{uq}+ \mathcal C(X)_{uq})\\
	&\qquad - \tfrac{1}{2} h^{lm} h^{jk} h^{sq} B_{jks} \big( \beta(X)_{ql;m} + \beta(X)_{qm;l} - \beta(X)_{lm;q}\big)\\
	&\qquad -\tfrac{1}{2} h^{lm} h^{jk} h^{sq} B_{jks}  \big( \mathcal C(X)_{ql;m} + \mathcal C(X)_{qm;l} - \mathcal C(X)_{lm;q} \big)\\
	&\qquad - \big( h^{lm} h^{ju} h^{qk} B_{jkl} (\beta(X)_{uq}+ \mathcal C(X)_{uq}) \big)_{;m}\\
	&\qquad + h^{lm} h^{jk} \big( - X_{l;kj}+ X^q h^{us} B_{kqs} B_{jlu} + X^s h^{uq} B_{lsq} B_{jku}\big)_{;m}\\
	& \qquad + h^{lm} h^{jk} X^s \big( g  \big(\mathcal{D}((\bar\nabla_{E_k}E_s)^{\perp})  ,  (\bar\nabla_{E_j} E_l )^{\perp} \big) + g \big((\bar\nabla _{E_j} E_k)^{\perp} , \mathcal{D}((\bar\nabla_{E_l}E_s)^{\perp} ) \big)\big)_{;m}\\
	&\qquad - h^{lm} h^{jk}\big( X^s \bar R_{jskl}  + ( \Gamma_{jk}^s )' \omega_{sl} \big)_{;m} \\
	&=  \mathcal E_1 (X) +\mathcal{E}_2(X) 
\end{align*}
where $\mathcal E_1 (X)$ and $\mathcal{E}_2(X)$ are as in the statement of the proposition. In attaining these expressions, we have expanded $\beta(X)_{ij} = X^s ( B_{isj} + B_{jsi} )$ and we have denoted  the components of the ambient Ricci tensor by $\bar R_{sl}$.  The point of arranging the outcome of the calculation in this way is because the term $\mathcal E_1(X)$ has the same form as the linearization of the Hamiltonian stationary Lagrangian differential operator at a Lagrangian submanifold while the term $\mathcal{E}_2(X)$ vanishes at a Lagrangian submanifold.
\end{proof} 

The next step in the calculation is to determine the decomposition $L^{(s)}_\rho (X) = \mathring L^{(s)}(X) + P^{(s)}_\rho(X)$ for $s = 1, 2$.  Of course, $L^{(1)}_\rho(X) = \dif (X \elbow \omega)$ according to the usual Poincar\'e formula and so 
$$P_{\rho}^{(1)}(X) = \dif(X\elbow \omega)- \dif(X\elbow \mathring\omega)= \dif(X\elbow (\omega-\mathring\omega)) \, .$$
For $P_\rho^{(2)}$, observe that $\mathcal E_1(X)$ has the same form as $\mathring L^{(2)}(X)$ and $\mathcal{E}_2(X)$ vanishes when $\rho = 0$.  Thus formally we can decompose
$$P^{(2)}_\rho(X) = \big(\mathcal E_1(X) - \mathring L^{(2)}(X) \big) + \mathcal E_2(X) \, .$$ 
We will not determine the precise form of the operator $\mathcal E_1(X) - \mathring L^{(2)}(X) $ since these details will not be needed in the sequel.  

\begin{cor}
	\label{cor:differenceop}
	The components of the operator $P_\rho$ are
	\begin{align*}
		P_{\rho}^{(1)}(X) &:= \dif(X\elbow (\omega-\mathring\omega)) \\
		P^{(2)}_\rho(X) &:= \big(\mathcal E_1(X) - \mathring L^{(2)}(X) \big) + \mathcal E_2(X)  
	\end{align*}
	with notation as in Proposition \ref{prop:pertlin}.
\end{cor}

\noindent We now obtain a corresponding decomposition $\boldsymbol{L}_\rho^{(s)} := \mathring{\boldsymbol{L}}^{(s)} +  \boldsymbol{P}_\rho^{(s)}$ where $\boldsymbol{P}_\rho^{(s)} := P_\rho^{(s)} \circ \mathcal X$.

\subsection{Estimates for the Perturbed Linearization}

The norms that will be used to estimate the various quantities involved in the proof of the Main Theorem will be the standard $C^{k,\alpha}$ norms; these will be taken with respect to the background metric $\mathring g$ when the quantity being estimated is defined in $\C^2$ and with respect to the induced metric $\mathring h$ when the quantity being estimated is defined on the submanifold $\Sigma_r$.  Note that these norms are equivalent to those defined by the metrics $g$ and $h$ and coincide with the norms used in the statement of the Main Theorem when the re-scaling of Section \ref{sec:scale} is reversed.  Begin with the following lemma.

\begin{lemma}  
\label{lemma:prelimest}

Let $\Sigma$ be a totally real submanifold of $\C^2$ equipped with the K\"ahler metric $g$.  Fix $\alpha \in (0,1)$ and $k \in \N$.  There is a constant $C$ independent of $\rho$ so that for all $X \in \Gamma (J(T\Sigma))$ and $W\in \Gamma(N\Sigma)$ the following estimates hold:
\begin{equation*}
	\begin{aligned}
		\|g-\mathring g\|_{C^{k,\alpha}(M)} &\leq  C\rho^2 \\ 
		\|\omega-\mathring \omega\|_{C^{k,\alpha}(M)} &\leq   C \rho^2\\
		\|B-\mathring B\|_{C^{k,\alpha}(\Sigma_r)} &\leq  C\rho^2 \\
		 \|H -\mathring H\|_{C^{k,\alpha}(\Sigma_r)} &\leq  C\rho^2\\
	 \end{aligned} \qquad \qquad
	 \begin{aligned}
		\| \nabla \cdot X -\mathring{\nabla} \cdot X \|_{C^{k,\alpha}(\Sigma_r)} &\leq  C \rho^2 \|X\|_{C^{k,\alpha}(\Sigma_r)}\\
		 \|\mathcal{C}(X)\|_{C^{k,\alpha}(\Sigma_r)}&\leq  C\rho^2 \|X\|_{C^{k+1,\alpha}(\Sigma_r)} \\
		 \|\mathcal{D}(W)\|_{C^{k,\alpha}(\Sigma_r)}&\leq  C\rho^2 \|W\|_{C^{k,\alpha}(\Sigma_r)} \\
		 \|\mathcal{E}_2(X)\|_{C^{k,\alpha}(\Sigma_r)}&\leq C \rho^2 \|X\|_{C^{k+2,\alpha}(\Sigma_r)} \, .
	\end{aligned}
\end{equation*}
Furthermore, the operator $\mathcal D$ vanishes if $\Sigma$ is Lagrangian.
\end{lemma}

\begin{proof}  The estimates mostly follow from the estimate of the K\"{a}hler potential $F_{\rho}(z, \bar z) := \frac{1}{2} \| z \|^2 + \rho^2 \hat F_\rho (z, \bar z)$, where $\hat F_\rho(z, \bar z) := \rho^{-4} \hat F(\rho z, \rho \bar z )$.  Recall that for any multi-index $\alpha$ the derivative $\partial^{\alpha} \hat F(\zeta, \bar \zeta)$ is $\mathcal{O}(\|\zeta\|^{4-\alpha})$ for $|\alpha|\leq 4$, and $\mathcal{O}(1)$ for $|\alpha|>4$.   This immediately gives the first two estimates.  The estimate on the symplectic second fundamental form comes from the following (and then immediately implies the estimate on the mean curvature one-form): 
\begin{align*}
	B(X,Y,Z)-\mathring B(X,Y,Z) &= \omega((\nabla_XY)^{\perp}, Z) - \mathring \omega ((\mathring \nabla _X Y)^{\perp}, Z) \, ,
\end{align*} 
where $(\nabla_X Y)^{\perp}= \nabla_X Y - h^{ij} g(\nabla_X Y, E_j) E_i$ and $(\mathring\nabla_X Y)^{\perp}= \mathring \nabla_X Y - \mathring h^{ij}\mathring g(\mathring \nabla_X Y, E_j)E_i$.  The above estimate of $(g-\mathring g)$ yields the analogous estimate of $\nabla - \mathring \nabla$, which together with the equation above then yields the estimate of $B-\mathring B$, as well as the estimate on the divergence.  

We now estimate $\mathcal{D}$, which, together with the above estimates, will also yield the estimate of $\mathcal{E}$, and thus complete the proof.  Let $W\in N_p \Sigma$ be a unit vector.  Recall from above that 
$$\mathcal{D}(W)=W-h^{ij}g(W, JE_j)JE_i = W+h^{ij}\omega (W, E_j)JE_i.$$
If we use the orthogonal decomposition of $W$ with respect to the metric $\mathring g$, denoting it as $W= \mathring W^\perp + \mathring W^{\|}$, then since $g(W, E_j)=0$, we have immediately $\mathring g(W, E_j)= \mathcal{O}(\rho^2)$.  Thus $\mathring W^\| = O(\rho^2)$.  Furthermore, since $\Sigma$ is Lagrangian for $\mathring \omega$, then $\mathring W^\perp = -\mathring h^{ij} \mathring \omega(W^\perp, E_j) J E_i= -\mathring h^{ij} \mathring \omega(W, E_j)J E_i.$  Thus $\mathcal{D}(W)-\mathring W^\| = \mathring W^\perp + h^{ij} \omega(W,E_j) JE_i = \mathcal{O} (\rho^2)$. 
\end{proof}

Based on these elementary estimates, we have the following estimates of $P_{\rho}$ and $\boldsymbol{P}_\rho$ on a totally real submanifold $\Sigma$. 

\begin{prop} Let $\Sigma$ be a totally real submanifold of $\C^2$ equipped with the K\"ahler metric $g$. Fix $k \in \N$ and $\alpha \in (0,1)$.  There is a constant $C$ independent of $\rho$ so that 
\begin{align*} 
	\|P_{\rho}^{(1)}(X)\|_{C^{k,\alpha}} &\leq  C \rho^{2} \|X\|_{C^{k+1,\alpha}}\\
	\|P_{\rho}^{(2)}(X)\|_{C^{k,\alpha}} &\leq  C \rho^{2} \|X\|_{C^{k+2,\alpha}}\\
	\|\boldsymbol{P}_{\rho}^{(1)}(u, v)\|_{C^{k,\alpha}} &\leq  C \rho^{2} \|(u, v)\|_{C^{k+1,\alpha} \times C^{k+1,\alpha} }\\
	\|\boldsymbol{P}_{\rho}^{(2)}(u, v)\|_{C^{k,\alpha}} &\leq  C \rho^{2} \|(u, v)\|_{C^{k+2,\alpha} \times C^{k+2,\alpha}} \, .
\end{align*}
\end{prop}

\section{Solving the Hamiltonian Stationary Lagrangian PDE}

\subsection{Outline}
\label{sec:functional}

In this final section of the paper, the equation $\bPhi_\rho(u,v) = (0,0)$ will be solved for all $\rho$ sufficiently small using a perturbative technique.  An initial difficulty that must be overcome is that it is not possible to find a suitable inverse for the linearized operator $\boldsymbol{L}_\rho := \Dif_{(0,0)} \bPhi_\rho$ with $\rho$-independent norm because the operator $\mathring {\boldsymbol{L}} := \Dif_ {(0,0)} \mathring \bPhi$ has a non-trivial, six-dimensional kernel and fails to be surjective since its adjoint has a seven-dimensional kernel.  This fact makes a three-step approach for solving $\bPhi_\rho(u,v) = (0,0)$ necessary.

\paragraph*{Step 1.} The first step is to solve a \emph{projected} problem wherein the difficulties engendered by the kernel and co-kernel of $\mathring {\boldsymbol{L}}$ are avoided.  Let $\mathcal K$ be the kernel of $\mathring {\boldsymbol{L}}$ and let $\mathcal K^\ast$ be the kernel of $\mathring {\boldsymbol{L}}^\ast$.  Let 
$$\pi: C^{2, \alpha} (\dif \Lambda^1 (\Sigma_r)) \times C^{0, \alpha} (\Sigma_r) \rightarrow \Big( C^{2, \alpha} ( \dif \Lambda^1(\Sigma_r)) \times C^{0, \alpha}(\Sigma_r)\Big) \cap [\mathcal K^\ast ]^\perp$$
be the $L^2$-orthogonal projection onto $[\mathcal K^\ast ]^\perp$ with respect to the volume measure induced from the Euclidean ambient metric and consider the operator 
$$\pi \circ \bPhi_\rho \big|_{\mathcal K^\perp} : \Big( C_0^{4,\alpha}(\Sigma_r) \times C_0^{4, \alpha}(\Sigma_r) \Big) \cap \mathcal K^\perp \rightarrow  \Big( C^{2, \alpha} ( \dif \Lambda^1(\Sigma_r)) \times C^{0, \alpha}(\Sigma_r)\Big) \cap [\mathcal K^\ast ]^\perp \, .$$
The first step is thus to solve $\pi \circ \bPhi_\rho \big|_{\mathcal K^\perp} (u,v) = (0,0)$.  The linearization of this new operator is $\pi \circ \boldsymbol{L}_\rho \big|_{\mathcal K^\perp} $ which is by definition invertible at $\rho = 0$.  This operator remains invertible for sufficiently small $\rho>0$, and it will be shown below that a solution of the non-linear problem 
$$\pi \circ \bPhi_\rho \big|_{\mathcal K^\perp} (u,v) = (0,0)$$
can be found.  We will denote the solution by $(u_\rho, v_\rho)$ and let $\tilde \Sigma_r(\mathcal U_p) := \mu_{\mathcal X (u_\rho, v_\rho)} (\Sigma_r(\mathcal U_p))$ be the perturbed submanifold generated by this solution; we will abbreviate this by $\tilde \Sigma_r$ when there is no cause for confusion. 

\paragraph*{Step 2.} The previous step shows that a solution $(u, v) := (u_\rho, v_\rho)$ of the projected problem on $\Sigma_r$ can always be found so long as $\rho$ is sufficiently small.  One should realize that the solution $(u_\rho, v_\rho)$ that has been found depends implicitly on the point $p \in M$ and the choice of unitary frame $\mathcal U_p$ at $p$ out of which $\Sigma_r$ has been constructed.  Moreover, this  dependence is smooth as a standard consequence of the fixed-point argument used to find $(u_\rho, v_\rho)$.  The solution is such that $\bPhi_\rho (u_\rho, v_\rho)$ is an \emph{a priori} non-trivial but small quantity that belongs to $\mathcal K^\ast$.  

In the second step of the proof of the Main Theorem, it will be shown that when an \emph{existence condition} is satisfied at the point $p \in M$, there exists $\mathcal U_p$ so that  $\bPhi_\rho (u_\rho, v_\rho)$ vanishes except for a component in the space $\mathrm{span}_{\R} \{ (0, 1) \}$.  We set this up as follows.  First, write  $\mathcal K^\ast = \mathrm{span}_{\R} \{ (0, 1) \} \oplus \mathcal K^\ast_0$ where $\mathcal K^\ast_0 := \mathrm{span}_{\R}  \{ f^{(1)} v^{(1)}, \ldots, f^{(6)} v^{(6)}\}$ and the $v^{(i)}$ are constant vectors determined in Corollary \ref{cor:cokernel}, normalized so that the second component $v_2^{(i)}=1$.  Therefore
$$\bPhi(u_\rho, v_\rho) = a (0,1) + \sum_{j=1}^6 b_j f^{(j)} v^{(j)} \qquad \mbox{for some } a, b_1, \ldots, b_6 \in \R$$
Now define a smooth mapping on the unitary 2-frame bundle $\mathbf{U}_2 (M)$ over $M$, given by
\begin{gather*}
	G_\rho : \mathbf{U}_2(M) \rightarrow \R^6 \\
	G_\rho (  \mathcal U_p ) := \big( I^{(1)}_\rho , \ldots, I^{(6)}_\rho \big)  
\end{gather*}
where
\begin{equation}
	\label{eqn:cokerproj}
	I_\rho^{(i)} (  \mathcal U_p ) := \int_{\Sigma_r }  \big( f^{(i)} - c^{(i)} \big) v^{(i)} \cdot \bPhi(u_\rho, v_\rho) \dif \mathrm{Vol}_{\Sigma_r }
\end{equation}
and $c^{(i)}$ has been chosen to ensure that $\int_{\Sigma_r} \big( f^{(i)} - c^{(i)} \big) \dif \mathrm{Vol}_{\Sigma_r } = 0$.  We now have 
$$I_\rho^{(i)} (  \mathcal U_p ) = \sum_{i=1}^6 b_i \int_{\Sigma_r} f^{(j)} f^{(i)} \dif \mathrm{Vol}_{\Sigma_r} 
$$
and would now like to find $\mathcal U_p$ so that $G_\rho (\mathcal U_p) \equiv 0$.  This will turn imply that $b_i = 0$ for all $i$ because the matrix whose coefficients are the integrals $\int_{\Sigma_r} f^{(j)} f^{(i)} \dif \mathrm{Vol}_{\Sigma_r}$ is an invertible matrix.  

The idea for locating a zero of $G_\rho$ is first to find $\mathcal U_p$ so that $G_\rho(\mathcal U_p)$ vanishes to lowest order in a Taylor expansion in powers of $\rho$, but in such a way that  $G_\rho$ remains locally surjective at this $\mathcal U_p$.  The implicit function theorem for finite-dimensional manifolds can then be invoked to find a nearby $ \mathcal U_{p'}$ for which $G_\rho ( \mathcal U_{p'}) \equiv 0$ exactly.

\paragraph*{Step 3.}  The previous step shows that the only non-vanishing component of $\nabla \cdot H \big( \tilde \Sigma_r  \big)$ is perhaps the projection of $\nabla \cdot H \big( \tilde \Sigma_r  \big)$ to $\mathrm{span}_{\R} \{ (0, 1) \}$.  But the divergence theorem can now be invoked to show that this component must vanish as well, thereby completing the proof of the Main Theorem.

\subsection{Estimates for the Approximate Solution}
\label{sec:estimates}

To begin, we must compute the size of $\| \bPhi_\rho(0,0) \|_{C^{2,\alpha} \times C^{0,\alpha}}$ which must be sufficiently small for the perturbation method of Step 1 to succeed.

\begin{prop}
	\label{prop:size}
	There is a constant $C>0$ independent of $\rho$ so that 
	$$\| \bPhi_\rho(0,0) \|_{C^{2,\alpha} \times C^{0, \alpha}}\leq C \rho^2 \, .$$
\end{prop}

\begin{proof}  
By Lemma \ref{lemma:torus}, we have $\mathring \bPhi(0,0)=(0,0)$.  By Lemma \ref{lemma:prelimest}, we have $\|\omega-\mathring \omega \|_{C^{2,\alpha}}\leq C\rho^2$.  Furthermore, by writing $$\nabla \cdot H = \mathring \nabla \cdot \mathring H + (\nabla-\mathring \nabla) \cdot \mathring H + \nabla \cdot (H-\mathring H),$$ we have 
$$\| \nabla \cdot H \|_{C^{0,\alpha}} \leq C \rho^2\|\mathring H\|_{C^{0,\alpha}} + \|H-\mathring H\|_{C^{1,\alpha}}\leq C\rho^2.$$
again using the estimates of Lemma \ref{lemma:prelimest}.
\end{proof}

\subsection{Solving the Projected Problem}

This section proves that Step 1 from the outline above can be carried out.

 \begin{thm}
 	\label{thm:projprob}
 	For every $\rho$ sufficiently small, there is a solution $(u_\rho, v_\rho) \in \Big( C_0^{4,\alpha}(\Sigma_r) \times C_0^{4,\alpha}(\Sigma_r)\Big) \cap \mathcal K^\perp$ that satisfies
	$$ \pi \circ \bPhi_{\rho} (u_\rho, v_\rho) = (0,0) \, .$$
	Moreover, the estimate $\|(u_\rho, v_\rho)\|_{C^{4,\alpha} \times C^{4,\alpha}} \leq C \rho^2$ holds.
 \end{thm}
 
 \begin{proof}
 	The solvability of the equation $\pi \circ \bPhi_\rho (u, v) = (0,0)$ is governed by the behaviour of the linearized operator $\pi \circ \boldsymbol{L}_\rho$ between the Banach spaces given in the statement of the theorem, as well as on the size of $\| \bPhi_{\rho}(0,0) \|_{C^{2,\alpha} \times C^{0,\alpha}}$, which we know to be $\mathcal O( \rho^2)$ by Proposition \ref{prop:size}.
	
	First, by standard elliptic theory, the operator $\pi \circ \mathring{\boldsymbol{L}}$ is invertible between $\mathcal K^\perp$ and $[\mathcal K^\ast]^\perp$ with the estimate
$$\| \pi \circ \mathring{\boldsymbol{L}}(u,v) \|_{C^{2,\alpha} \times C^{0,\alpha}} \geq C \| (u,v) \|_{C^{4,\alpha} \times C^{4, \alpha}} $$
where $C$ is a constant independent of $\rho$.  Consequently, if $\rho$ is sufficiently small, then the operator $\pi \circ \boldsymbol{L}_\rho$ is uniformly injective with the estimate
$$\| \pi \circ \boldsymbol{L}_\rho(u,v) \|_{C^{2,\alpha} \times C^{0,\alpha}} \geq \frac{C}{2} \| (u,v) \|_{C^{4,\alpha} \times C^{4, \alpha}} \, .$$
Hence by perturbation, the operator $\pi \circ \boldsymbol{L}_\rho$ is also surjective onto $[\mathcal K^\ast]^\perp$ and the inverse is bounded above independently of $\rho$.  

The remainder of the proof uses the contraction mapping theorem.  First, write 
$$\pi \circ \bPhi_\rho (u,v) := \pi \circ \bPhi_\rho(0,0) + \pi \circ \boldsymbol{L}_\rho (u, v) + \pi \circ \boldsymbol{Q}_\rho (u, v)$$
where $\boldsymbol{Q}_\rho$ is the quadratic remainder (in $u$ and $v$) $\bPhi_\rho$.  It is fairly straightforward to show that $\boldsymbol{Q}_\rho$ satisfies the estimate
$$\| \boldsymbol{Q}_\rho (u_1, v_1) - \boldsymbol{Q}_\rho  (u_2, v_2) \|_{C^{2,\alpha}\times C^{0,\alpha}} \leq C \| (u_1 + u_2, v_1+ v_2)\|_{C^{4,\alpha}\times C^{4,\alpha}} \|(u_1- u_2 , v_1- v_2) \|_ {C^{4,\alpha}\times C^{4,\alpha}} $$
for some constant $C$ independent of $\rho$, provided $\rho$ is sufficiently small.  This is because such an estimate is certainly true for the quadratic remainder of $\mathring \bPhi$.  Now let $\boldsymbol{L}^{-1}_\rho: [\mathcal K^\ast]^\perp\rightarrow \mathcal K^\perp$ denote the inverse of $\boldsymbol{L}_\rho$ onto $\mathcal K^\perp$.  By proposing the \emph{Ansatz}  $(u,v) := \boldsymbol{L}^{-1}_\rho \big( - (w, \xi) - \pi \circ \bPhi_{\rho}(0,0) \big)$, for $(w,\xi)\in [\mathcal K^\ast]^\perp$, the equation $\pi\circ \bPhi_\rho (u,v) = (0,0)$ becomes equivalent to the fixed-point problem for the map 
	$$ \mathcal N_\rho : (w, \xi) \mapsto \pi \circ \boldsymbol{Q}_\rho \big( \boldsymbol{L}^{-1}_\rho \big( - (w, \xi) - \pi \circ \bPhi_{\rho}(0,0) \big) \big) 
	$$
on $[\mathcal K^\ast]^\perp$.  For small enough $\rho$, the non-linear mapping $(w, \xi) \mapsto 	\pi \circ \boldsymbol{Q}_\rho \big( \boldsymbol{L}^{-1}_\rho \big( - (w, \xi) - \pi \circ \bPhi_{\rho}(0,0) \big) \big)$ verifies the estimates required to find a fixed point in a closed ball $B\subset  [\mathcal K^\ast]^\perp$ of radius equal to $\| \bPhi_\rho(0,0) \|_{C^{2,\alpha} \times C^{0,\alpha}} = \mathcal O(\rho^2)$ by virtue of the $\rho$-independent estimates that have been found for $\boldsymbol{L}^{-1}_\rho$ and $\boldsymbol{Q}_\rho$.  For example, for $(w,\xi)\in B$, 
$$\| \mathcal N_\rho (w, \xi) \|_{C^{2,\alpha}\times C^{0,\alpha}}\leq C \| \bPhi_\rho(0,0) \|^2_{C^{2,\alpha} \times C^{0,\alpha}} \leq \| \bPhi_\rho(0,0) \|_{C^{2,\alpha} \times C^{0,\alpha}} $$
for $\rho$ small enough; hence the set $B$ is mapped to itself under $\mathcal N_\rho$.   Furthermore, $\mathcal N_\rho$ is a contraction on $B$ as a result of the biinear estimate on $\boldsymbol{Q}_\rho$ given above.   Consequently, $\mathcal N_\rho$ must have a fixed point $ (w, \xi) \in B $ which thus satisfies $\|(w, \xi)\|_{C^{2,\alpha} \times C^{0,\alpha}} \leq C \rho^2$ for some constant $C$ independent of $\rho$.  The desired estimate follows.
\end{proof}
 
\begin{rmk}   
	The solution $(u_\rho, v_\rho)$ is in fact smooth by elliptic regularity theory and the estimate $\|(u_\rho, v_\rho)\|_{C^{k,\alpha}\times C^{k,\alpha}} \leq C \rho^2$ holds for all $k \in \N$, where $C$ is independent of $\rho$.
\end{rmk}

\subsection{Derivation of the Existence Condition}

The remainder of the proof begins with a more careful investigation of the integrals \eqref{eqn:cokerproj} for all choices of $f$ spanning $\mathcal K_0^\ast$.  Recall that such $f$ come from translation and $U(2)$-rotation in the local coordinates at the point $p$; one can thus construct a basis for $\mathcal K_0^\ast$ as follows.  Let $(U, \tau) \,\cdot $ denote the motion of $\C^2$ given by $z \mapsto U(z) + \tau$ where $U \in U(2)$ and $\tau \in \C^2$.  Then we consider the six-dimensional parameter family of motions of $M$ given by
$$\mathcal R := \big\{ \big( \exp( \mi \tau_5 K_1 + \mi \tau_6 K_2), \tau \big) \cdot \, : \, \tau_5, \tau_6 \in \R \mbox{ and } \tau  := (\tau_1, \ldots, \tau_4) \in \R^4 \big\} $$
where 
$$K_1 := \left( \begin{matrix} 0 & 1 \\ 1 & 0 \end{matrix} \right) \qquad \mbox{and} \qquad K_2 := \left( \begin{matrix} 0 & \mi \\ -\mi & 0 \end{matrix} \right)$$
are elements in the Lie algebra of $U(2)$ that generate all non-trivial $U(2)$-rotations of $\Sigma_r$.  Note that the orbit of $\mathcal U_p$ under a small neighbourhood of the identity in $\mathcal R$ projects diffeomorphically onto a neighbourhood of $[\mathcal U_p] \in \mathbf{U}_2(M) / \mathit{Diag}$.  Denote by $\mu_t^{(i)}$ for $i = 1, \ldots 6$ those motions which correspond to $\tau_i = t$ and $\tau_{i'} = 0$ for $i' \neq i$.  Note that each $\mu^{(i)}$ is Hamiltonian with respect to the Euclidean K\"ahler structure, with $J \mathring \nabla f^{(i)} := \frac{\dif}{\dif t} \mu_t^{(i)} \big|_{t=0} $.  Moreover the restriction of $f^{(i)}$ to $\Sigma_r$ belongs to $\mathcal K_0^\ast$.  Indeed, the translations $\mu_t^{(1)}, \cdots , \mu_t^{(4)}$ yield the functions $\cos(\theta^s)$ and $\sin(\theta^s)$ for $s=1,2$ while the $U(2)$-rotations $\mu_t^{(5)}$ and $\mu_t^{(6)}$ yield the functions $\sin(\theta^1 - \theta^2)$ and $\cos (\theta^1 - \theta^2)$.  
   
We can relate the integrals $I^{(i)}_\rho (\mathcal U_p)$ to the ambient geometry of $M$ to lowest order in $\rho$ using the first variation formula along with Stokes' theorem.   Let $v^{(i)} := ( v^{(i)}_1 ,1 )$ and note that $v^{(5)}_1 = v^{(6)}_1 = 0$. 

\begin{lemma} The following formula holds. 
\begin{equation}
	\label{eqn:Iformula}
	I^{(i)}_\rho (\mathcal U_p) =  \left. \frac{\dif}{\dif t} \mathit{Vol} \big(\mu_t^{(i)} ( \Sigma_r  )\big) \right|_{t=0} +  v^{(j)}_1 \int_{\Sigma_r} f^{(j)} \cdot  ( \omega - \mathring \omega) + \mathcal O(\rho^4) \, .
\end{equation} 
\end{lemma}

\begin{proof} After a careful computation, we find
\begin{align*}
	\int_{\Sigma_r}    (f^{(j)} - c^{(j)})  \bPhi(u_\rho, v_\rho) \cdot v^{(j)}  \dif \mathrm{Vol}_{\Sigma_r} &\\[1ex]
	&\hspace{-30ex} =    \int_{\Sigma_r} \nabla \cdot H(\Sigma_r)   (f^{(j)} - c^{(j)}) \dif \mathrm{Vol}_{\Sigma_r}  +  v^{(j)}_1 \int_{\Sigma_r} (f^{(j)} - c^{(j)}) ( \omega - \mathring \omega)  \\
	&\hspace{-30ex} \qquad +  \int_{\Sigma_r} (f^{(j)} - c^{(j)})  \mathring{\boldsymbol{L}} (u_\rho, v_\rho) \cdot v^{(j)} \dif \mathrm{Vol}^\circ_{\Sigma_r}  \\
	&\hspace{-30ex} \qquad  +  \int_{\Sigma_r} (f^{(j)} - c^{(j)}) \boldsymbol{L}_\rho (u_\rho, v_\rho) \cdot v^{(j)} \big( \dif \mathrm{Vol}_{\Sigma_r} - \dif \mathrm{Vol}^\circ_{\Sigma_r} \big) \\
	&\hspace{-30ex} \qquad + \int_{\Sigma_r} (f^{(j)} - c^{(j)}) \boldsymbol{P}_\rho (u_\rho, v_\rho) \cdot  v^{(j)} \dif \mathrm{Vol}_{\Sigma_r} +  \int_{\Sigma_r} (f^{(j)} - c^{(j)}) \boldsymbol{Q}_\rho (u_\rho, v_\rho) \cdot  v^{(j)} \dif \mathrm{Vol}_{\Sigma_r} \\[1ex]
	&\hspace{-30ex}=    \left. \frac{\dif}{\dif t} \mathit{Vol} \big(\mu_t^{(j)} ( \Sigma_r  )\big) \right|_{t=0} +  v^{(j)}_1 \int_{\Sigma_r} (f^{(j)} - c^{(j)}) ( \omega - \mathring \omega)  \\
	&\hspace{-30ex}\qquad +  \int_{\Sigma_r} (f^{(j)} - c^{(j)})  \mathring{\boldsymbol{L}} (u_\rho, v_\rho) \cdot v^{(j)} \dif \mathrm{Vol}^\circ_{\Sigma_r}  + \mathcal O(\rho^4) \, .
\end{align*}
Here we have used the expansion $\bPhi_\rho(u_\rho, v_\rho) =  \bPhi_\rho(0, 0)+ \boldsymbol{L}_{\rho}(u_\rho, v_\rho) +  \boldsymbol{Q}_{\rho}(u_\rho, v_\rho)$, where $\boldsymbol{L}_{\rho}= \mathring{\boldsymbol{L}}+\boldsymbol{P}_\rho$ and $\boldsymbol{Q}_\rho$ is the quadratic remainder of the operator $\bPhi_\rho$, along with the following facts:
\begin{itemize}
	\item $\| (u_\rho, v_\rho) \|$ and $\| \mathring{\boldsymbol{L}}^{(2)}(u_\rho, v_\rho) \|_{C^0}$ and $  \| \nabla \cdot H (\Sigma_r) \|_{C^0}$ are all $\mathcal O(\rho^2)$ 
	\item  $\| \boldsymbol{P}_\rho (u_\rho, v_\rho) \|_{C^0} \leq C \rho^2 \| (u_\rho, v_\rho) \|_{C^{2,\alpha}\times C^{2,\alpha}}= \mathcal O(\rho^4)$
	\item  $\| \boldsymbol{Q}_\rho (u_\rho, v_\rho) \|_{C^0} \leq C \| (u_\rho, v_\rho) \|_{C^{4,\alpha}\times C^{4,\alpha}}^2 = \mathcal O(\rho^4)$
	\item the difference between any of the volume forms appearing above is $\mathcal O(\rho^2)$
	\item $\int_{\Sigma_r} f^{(i)} \dif \mathrm{Vol}_{\Sigma_r}^\circ = 0$ which implies $|c^{(i)} | = \mathcal O(\rho^2)$
\end{itemize}
along with Stokes' Theorem.  To complete the proof of the lemma, we note that the second term vanishes since $(f^{(j)} - c^{(j)}) v^{(j)} $ belongs to the kernel of $\mathring{\boldsymbol{L}}^\ast$.
\end{proof}

\noindent Now, let $\{ \Xi^{(1)}, \ldots, \Xi^{(6)}\}$  be the vectors in $T_{[\mathcal{U}_p]} \big( \mathbf{U}_2(M) /  \mathit{Diag} \big)$ corresponding to motions $\{ \mu_t^{(1)}, \ldots, \mu_t^{(6)} \}$ above. 

\begin{prop} 
	\label{prop:expansion}
	Define the smooth mapping
	\begin{gather*}
	{\mathcal F}_r : \mathbf{U}_2(M) / \mathit{Diag} \rightarrow \R \\
	{\mathcal F}_r ([ \mathcal U_p ] ) := r_1^2 R^{\C}_{1 \bar 1}(p) + r_2^2 R^{\C}_{2 \bar 2}(p) 
\end{gather*}
where the components of the complex Ricci curvature $R^{\C}_{1 \bar 1}$ and $R^{\C}_{2 \bar 2}$ are computed with respect to the chosen frame. Then the mapping $G_\rho : \mathbf{U}_2 (M) \rightarrow \R^6$ defined by $G_\rho(\mathcal U_p) :=  \big( I_\rho^{(1)}(\mathcal U_p), \ldots, I_\rho^{(6)}(\mathcal U_p) \big)$ satisfies 
\begin{equation}
	\label{eqn:finalest}
	G_\rho (\mathcal U_p) =  4 \pi^2 r_1 r_2 \, \rho^2 D \mathcal F_r ([\mathcal U_p])\cdot (\Xi^{(1)}, \ldots, \Xi^{(6)}) + \mathcal O(\rho^3) \, .
\end{equation}
\end{prop}

\begin{proof}
We expand the terms appearing in \eqref{eqn:Iformula}.  We begin with the derivative of the volume since it is the more involved quantity.  We have 
\begin{align*}
	\mathit{Vol} \big(\mu_t^{(i)}( \Sigma_r  ) \big) &= \mathit{Vol} \big((U_t, \tau_t ) \cdot \Sigma_r  \big) = \int_{(U_t, \tau_t) \cdot \Sigma_r} \big( \det (h_{F_\rho, t}) \big)^{1/2} \, \dif \theta^1 \wedge \dif \theta^2  + \mathcal O(\rho^4) 
\end{align*}
where $(U_t, \tau_t)\cdot$ is the motion corresponding to $\mu_t^{(i)}$ while $h_{F_\rho, t}$ is the induced metric of $(U_t, \tau_t) \cdot \Sigma_r$ with respect to the K\"ahler metric whose K\"ahler potential is $F_\rho$.  But 
\begin{align*}
	\int_{(U_t, \tau_t) \cdot \Sigma_r} \big( \det (h_{F_\rho, t}) \big)^{1/2} \, \dif \theta^1 \wedge \dif \theta^2  
	& = \int_{\Sigma_r} \big( \det (h_{F_\rho^t}) \big)^{1/2} \, \dif \theta^1 \wedge \dif \theta^2
\end{align*}
where $h_{F_\rho^t}$ is the induced metric of $\Sigma_r$ with respect to the K\"ahler metric whose K\"ahler potential is $F_\rho^t := F_\rho \circ (U_t, \tau_t)$, as can be checked fairly easily.   Therefore to complete the calculation, one must find the first few terms of the Taylor series of  $\big( \det (h_{F_\rho^t}) \big)^{1/2}$ in $\rho$ and allow the integration over the torus to pick out certain terms.

To this end, note that if $f : \C^2 \rightarrow \R$ is a real-valued function then elementary Fourier analysis shows that its restriction to the torus satisfies 
\begin{equation}
	\label{eqn:fourier}
	\int_{\T^2} f(\me^{\mi \theta^1}, \me^{\mi \theta^2}) \, d\theta^1 \wedge d \theta^2 = 4 \pi^2 \big(f(0) +  r_1^2 f_{,11}(0) + r_2^2 f_{,22}(0) \big) + Q^{(4)}(r_1, r_2)
\end{equation}
where  $Q^{(4)}$ consists only of terms coming from fourth and higher-order Fourier coefficients of $f \big|_{\T^2}$. This formula can be seen by writing $f(z, \bar z) :=  f(0) + \frac{\partial f}{\partial z^k}(0) z^k + \frac{\partial f}{\partial \bar z^k}(0) \bar z^k + \cdots$ and substituting $\bar z^k = r^k \me^{\mi \theta^k}$; the integration over the torus then causes all odd-order combinations of $z^k$ and $\bar z^k$ to vanish while giving exactly the terms in \eqref{eqn:fourier} at order two.   To apply this to the calculation at hand, first compute 
$$F_\rho^t (z, \bar z) := \frac{1}{2} \| U_t( z) + \tau_t \|^2 + \rho^2 (\hat F_\rho^t)^{(4)}(z, \bar z) + \rho^3 \mathcal O(\| z \|^5)$$
where $(\hat F_\rho^t)^{(4)}(z, \bar z)$ is the $\mathcal O(\| z\|^4)$ term in the Taylor series expansion of $\hat F_\rho \circ (U_t, \tau_t)$.   Now let $Q^{(3)}(z_1, z_2)$ denote a cubic polynomial in its arguments and observe
$$h_{F_\rho^t} = \Re \sum_{a,b} \Big( r_a^2 \delta_{ab} + \rho^2 (\hat F_\rho^t)^{(4)}_{ ,a \bar b}\, \, r_a r_b \me^{\mi (\theta^a - \theta^b)} + \rho^3 Q^{(3)}( r_1 \me^{\mi \theta^1} , r_2 \me^{\mi \theta^2}) + \mathcal O(\rho^4) \Big) \dif \theta^a \otimes \dif \theta^b \, .$$
The $\mathcal O(\rho^4)$ term is quartic and higher in $r_k \me^{\mi \theta^k}$. Integrating and taking advantage of the fact that  the cubic terms in the expansion of $\big( \det (h_{F_\rho^t}) \big)^{1/2}$ must vanish we can express
\begin{align}
	\label{eqn:integral}
	\int_{\Sigma_r} \big( \det (h_{F_\rho^t}) \big)^{1/2} \dif \theta^1 \wedge \dif \theta^2 &= \int_{\Sigma_r} r_1 r_2 \Big( 1 + \frac{\rho^2}{2 } \sum_c  ( \hat F_\rho^t)^{(4)}_{, c \bar c} \Big) \dif \theta^1 \wedge \dif \theta^2  + \mathcal O(\rho^4).
	\end{align}
Next, we write the first few terms of the Fourier expansion of the integrand (via the Taylor expansion) and integrate these to re-write the $\mathcal O(\rho^2)$ part of \eqref{eqn:integral} as 
\begin{align*}
	&r_1 r_2  \int_{\Sigma_r}\rho^2 \sum_{c, u, v} r_u r_v  \, \Re \big( (\hat F_\rho^t)_{, c \bar c u v} (0) \me^{\mi(\theta^u + \theta^v)} + (\hat F_\rho^t)_{, c \bar c u \bar v} (0) \me^{\mi(\theta^u - \theta^v)} \big) \dif \theta^1 \wedge \dif \theta^2 \ .
\end{align*}
Performing this integral yields
\begin{align*}
	\mathit{Vol} \big( (U_t, \tau_t ) \cdot \Sigma_r \big) &= 4 \pi^2 r_1 r_2 \Big( 1 + \rho^2 \sum_{c, u} r_u^2 (\hat F_\rho^t)_{, c \bar c u \bar u}(0) \Big) + \mathcal O(\rho^4) \\
	&\hspace{-10ex}= 4 \pi^2 r_1 r_2 \Big( 1 +  \rho^2 \big( r_1^2 (\hat F_\rho^t) _{, 1 1 \bar 1 \bar 1}(0) +  (r_1^2 + r_2^2) (\hat F_\rho^t)_{, 1 2 \bar 1 \bar 2}(0) + r_2^2 (\hat F_\rho^t)_{, 2 2 \bar 2 \bar 2}(0) \big) \Big) + \mathcal O(\rho^4)
\end{align*}
after explicitly expanding the sums over $c$ and $u$.  Therefore the lowest-order term in the expansion of $\frac{\dif}{\dif t} \mathit{Vol} \big( (U_t, \tau_t ) \cdot \Sigma_r \big) \big|_{t=0}$ in $\rho$ is
\begin{equation}
	\label{eqn:excond}
	\left. \frac{\dif}{\dif t} \right|_{t=0}  4 \pi^2 r_1 r_2 \Big( r_1^2 ( \hat F_\rho^t)_{, 1 1 \bar 1 \bar 1}(0) +  (r_1^2 + r_2^2) (\hat F_\rho^t)_{, 1 2 \bar 1 \bar 2}(0) + r_2^2 ( \hat F_\rho^t)_{, 2 2 \bar 2 \bar 2}(0) \Big) \, .
\end{equation}
Using \eqref{eqn:cxcurvature}, the expression \eqref{eqn:excond} can be re-phrased in terms of the complex Ricci curvature of $M$ as 
\begin{equation*}
	\left. \frac{\dif}{\dif t} \right|_{t=0}  4 \pi^2 r_1 r_2 \Big(r_1^2 R^{\C}_{ 1\bar 1 } ((U_t, \tau_t) \cdot p) + r_2^2 R^{\C}_{2\bar 2} ((U_t, \tau_t) \cdot p) \Big) \, .
\end{equation*}

We now turn to the difference of symplectic forms term.  The expression $\omega - \mathring \omega$ has leading order $\rho^2$ and the leading order part is an antisymmetric 2-tensor whose coefficients are homogeneous quadratic polynomials in $z$ and $\bar z$.  Pulling this back to $\Sigma_r$ yields an expression whose leading order part is a homogeneous fourth degree polynomial in $\cos(\theta^s)$ and $\sin(\theta^s)$ for $s = 1, 2$.  Multiplying this by $f^{(i)}$ for $i = 1, 2, 3$ or $4$ produces a fifth degree polynomial in these quantities.  This always integrates to zero over the torus.  Note that it is not necessary to consider the integrals against $f^{(5)}$ or $f^{(6)}$ since $v^{(5)}_1 = v^{(6)}_1 = 0$.  Hence the magnitude of $v^{(i)} \int_{\Sigma_r} f^{(i)} \cdot ( \omega - \mathring \omega)$ is determined by the next-to-leading terms in the expansion of $\omega - \mathring \omega$.  These are all $\mathcal O(\rho^3)$.  Expression \eqref{eqn:finalest} follows.
\end{proof}

\subsection{The Proof of the Main Theorem}

In this section, we conclude the proof of the Main Theorem by showing that if the mapping ${\mathcal F}_r $ has a non-degenerate critical point $[\mathcal{U}_p]$ in $\mathbf{U}_2(M) / \mathit{Diag} $, then $\tilde \Sigma_r(\mathcal U_p)$ can be further perturbed into an exactly Hamiltonian stationary Lagrangian submanifold.  This will then complete the proof of the Main Theorem.

\begin{thm}
	\label{thm:cokerprob}
	Suppose $[\mathcal U_p]$ is a non-degenerate critical point of the functional $\mathcal F_r$ defined in the previous section.  If $\rho$ is sufficiently small, then there is $ \mathcal U_{p'}$ near $\mathcal U_p$ so that the submanifold $\tilde \Sigma_r(\mathcal U_{p'}) $ that was obtained via Theorem \ref{thm:projprob} from the torus $\Sigma_r ( \mathcal U_{p'})$ is a Hamiltonian stationary Lagrangian submanifold.  The distance between $\mathcal U_p$ and $ \mathcal U_{p'}$ is $\mathcal O(\rho^2)$.
\end{thm}

\begin{proof}
	We must to find $\mathcal U_p$ so that $G_\rho(\mathcal U_p)$ vanishes identically.  But the estimate of Proposition \ref{prop:expansion} says that 
$$G_\rho (\mathcal U_p) = 4 \pi^2 r_1 r_2 \,  \rho^2 D \mathcal F_r ([\mathcal U_p])\cdot (\Xi_{(1)}, \ldots, \Xi_{(6)}) + \mathcal O(\rho^4) \, .$$
Suppose now that $D \mathcal F_r ([\mathcal U_p]) = 0$ and $D^2 \mathcal F_r([\mathcal U_p])$ is non-degenerate.  Since the norm of the inverse of $D^2 \mathcal F_r([\mathcal U_p])$ must be bounded above by a constant independent of $\rho$, then the implicit function theorem for maps between finite-dimensional manifolds implies that it is possible to find a neighbouring $\mathcal U_{p'}$ so that $G_\rho( \mathcal U_{p'}) \equiv 0$ provided $\rho$ is sufficiently small.  Furthermore the distance between $ \mathcal U_{p}$ and $\mathcal U_{p'}$ as points in $\mathbf{U}_2 (M)$ is $\mathcal O(\rho^2)$, which is a consequence of the fact that the error term in the equation $G_\rho(\mathcal U_p) = 0$ is $\mathcal O(\rho^4)$.  As indicated above, this now implies that $\nabla \cdot H \big( \tilde \Sigma_r  \big)$ is constant.  Then the divergence theorem implies that it must vanish.  
\end{proof}

\small 
\bibliography{tori}
\bibliographystyle{amsplain}

\end{document}